\documentclass{article}

\usepackage{arxiv}

\usepackage[utf8]{inputenc} 
\usepackage[T1]{fontenc}    
\usepackage{hyperref}       
\usepackage{url}            
\usepackage{booktabs}       
\usepackage{amsfonts}       
\usepackage{nicefrac}       
\usepackage{microtype}      
\usepackage{amsmath}
\usepackage{amsthm}
\usepackage{cleveref}       
\usepackage{lipsum}         
\usepackage{graphicx}
\usepackage{natbib}
\usepackage{doi}
\usepackage{enumitem}
\usepackage{pgfplots}
\usepackage{tabularray}
\usepackage{filecontents}
\usepackage{multirow}
\usepackage{subcaption}
\usepackage{caption}
\usepackage{float}
\usepackage{algorithm}
\usepackage{algorithmic}

\newtheorem{theorem}{Theorem}
\newtheorem{lemma}{Lemma}
\newtheorem{corollary}{Corollary}
\newtheorem{remark}{Remark}

\title{A Space-Time Dual-Pairing Summation-by-Parts Framework for Forward and Adjoint Wave Equations}

\newif\ifuniqueAffiliation
\uniqueAffiliationtrue

\ifuniqueAffiliation 
\author{ Kenny Wiratama\\
	Department of Mathematical Sciences\\
	Ulsan National Institute of Science and Technology\\
	Ulsan, South Korea \\
	\texttt{kenny.wiratama@unist.ac.kr} \\
	\And
	Kenneth Duru \\
	Department of Mathematical Sciences\\
	The University of Texas at El Paso\\
	El Paso, Texas, United States of America \\
    \\
    \And
    Yunho Kim \\
	Department of Mathematical Sciences\\
	Ulsan National Institute of Science and Technology\\
	Ulsan, South Korea \\
}
\else
\usepackage{authblk}

\newbox{\orcid}\sbox{\orcid}{\includegraphics[scale=0.06]{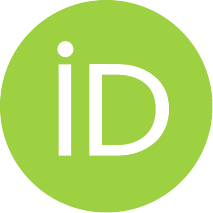}} 
\author[1]{%
	\href{https://orcid.org/0000-0000-0000-0000}{\usebox{\orcid}\hspace{1mm}David S.~Hippocampus\thanks{\texttt{hippo@cs.cranberry-lemon.edu}}}%
}
\author[1,2]{%
	\href{https://orcid.org/0000-0000-0000-0000}{\usebox{\orcid}\hspace{1mm}Elias D.~Striatum\thanks{\texttt{stariate@ee.mount-sheikh.edu}}}%
}
\affil[1]{Department of Computer Science, Cranberry-Lemon University, Pittsburgh, PA 15213}
\affil[2]{Department of Electrical Engineering, Mount-Sheikh University, Santa Narimana, Levand}
\fi

\renewcommand{\shorttitle}{A Space-Time DP-SBP Framework for Forward and Adjoint Wave Equations}

\hypersetup{
pdftitle={A Space-Time Dual-Pairing Summation-by-Parts Framework for Forward and Adjoint Wave Equations},
pdfsubject={MSC 68Q25, 68R10, 68U05},
pdfauthor={Kenny Wiratama, Kenneth Duru, Yunho Kim},
pdfkeywords={dual-pairing summation-by-parts, adjoint consistency, wave propagation, high order methods, fully discrete energy estimates},
}

\begin{document}

\maketitle

\begin{abstract}
In this paper, we propose the first of its kind space-time dual-pairing summation by parts (DP-SBP) numerical framework for forward and adjoint wave propagation problems. This novel approach enables us to achieve spatial and temporal high order accuracy while naturally introducing dissipation in time. Within this framework, initial and boundary conditions are weakly imposed using the simultaneous approximation term (SAT) technique. Fully discrete energy estimates are derived, ensuring the stability of the resulting numerical scheme. Furthermore, the proposed space-time numerical framework allows us to construct adjoint consistent fully discrete numerical approximations, which can be applied to solve inverse wave propagation problems. We provide numerical experiments in one and two spatial dimensions to verify the theoretical analysis and demonstrate convergence of numerical errors.
\end{abstract}


\par

\section{Introduction}

It is widely recognized that robust and high order accurate volume discretizations for complex geometries are critical to ensure reliable and optimal numerical solutions for most wave applications. Examples include acoustic and elastic wave propagation, ultrasound and photo-acoustic imaging, radar sensing, and nondestructive evaluation of materials and structures. Particularly for forward simulations,  numerical approximations should be provably stable, geometrically flexible, and high order accurate. Beyond forward simulations, there are many applications in the form of inverse problems where unknown wave states or material properties are inferred from measurements or observational data {\cite{cindea2015inverse,burman2020finite,burman2021space}}. However, the ill-posed nature of inverse problems, which makes them theoretically and computationally challenging, necessitates additional requirements on the numerical schemes. In particular, a numerical approximation must admit a consistent and energy stable adjoint as well as stability and high order accuracy so that sensitivities and gradients can be computed reliably and efficiently. Designing numerical methods that satisfy all of these requirements is a nontrivial task. 

Motivated by these considerations, the primary purpose of the current study is to develop a fully discrete numerical approximation for wave problems that is high order accurate, provably stable, and adjoint consistent. The proposed method will ensure optimal numerical solutions for both forward and inverse wave problems commonly encountered in several applications.

A well-known systematic way to construct provably stable and high order accurate numerical methods is provided by the summation-by-parts (SBP) numerical framework \cite{carpenter1994time, mattsson2004summation, mattsson2012summation, DELREYFERNANDEZ2014171, duru2014accurate, wang2017convergence}. SBP operators satisfy a discrete analogue of integration by parts, enabling provable stability through the derivation of discrete energy estimates analogous to those obtained from the continuous energy analysis. An extension of SBP, called dual-pairing SBP (DP-SBP) \cite{mattsson2017diagonal, duru2022dual,DOVGILOVICH2015176}, provides more flexibility to the numerical framework, leading to better accuracy for wave propagation problems \cite{williams2024full}.  The  DP-SBP  operators consist of forward and backward finite difference stencils that together obey  the SBP property. However, the existing literature primarily addresses semi-discrete formulations, and the extension of the SBP framework in time has received relatively little attention, aside from a few notable exceptions \cite{nordstrom2013summation, nordstrom2016summation, ranocha2021new}. In particular, the DP-SBP framework has not been extended in time. For the wave equation in second order form, it is not that obvious how to combine the DP-SBP operators to consistently approximate the first and second derivatives while ensuring numerical stability.

Another important concern arises in the context of inverse problems related to wave propagation. In particular, we consider inverse problems that can be formulated as PDE-constrained optimization problems of the form
{
\begin{equation}\label{pdeconstrainedopt}
\min_m \mathcal{J}(u(m)),
\quad \text{subject to } L(u(m),m)=S(m).
\end{equation}
}%
{where $\mathcal{J}(u(m)) = \int_0^T \Phi(u(m))\,dt$ is the objective functional, $\Phi$ is a convex and differentiable time-integrand functional, $m$ is the control parameter, $L(\cdot,m)$ is a linear wave differential operator for each fixed parameter $m$, $S(m)$ is a given source term, and $u(m)$ is the state variable satisfying appropriate initial and boundary conditions in a given domain.} Several works have applied the SBP numerical framework to solve PDE-constrained optimization problems  {\cite{hicken2011superconvergent,berg2012superconvergent,bader2023modeling, eriksson2024acoustic,stiernstrom2024adjoint}} similar to \eqref{pdeconstrainedopt} with analysis conducted at the semi-discrete level. One notes that an essential ingredient for reliable gradient-based numerical optimization in this setting is adjoint consistency \cite{hicken2011role}. When boundary conditions are weakly imposed using the simultaneous approximation term (SAT) method, the penalty parameters must be chosen carefully to ensure that the resulting numerical scheme is not only stable but also adjoint consistent. A numerical discretization is adjoint consistent if the adjoint problem induced by the discretization is a consistent approximation of the continuous adjoint problem \cite{hicken2014dual}. This property allows accurate evaluations of the objective functional, and efficient and reliable computations of its gradients. However, the time discretization of a semi-discrete adjoint problem may not coincide with the adjoint of the corresponding fully discrete forward problem \cite{WaltherandGriewank1999,Walther2007}.

Therefore, we contend that ensuring fully discrete adjoint consistency when time is discretized will advance the theory, minimize gradient errors, and improve the accuracy of numerical reconstructions in inverse problems. To date, the SBP numerical framework in time has not been applied to solve the PDE-constrained optimization problem \eqref{pdeconstrainedopt}, let alone the DP-SBP numerical framework.

In this work, we extend the DP-SBP framework in time to construct fully discrete numerical approximations for the second order wave equation. The initial and boundary conditions are weakly imposed using the SAT method, and by carefully choosing the penalty parameters, we derive semi-discrete and fully discrete energy estimates analogous to the continuous energy analysis. Furthermore, we demonstrate that both the semi-discrete and fully discrete approximations are adjoint consistent.   Finally, we use the fully discrete adjoint state technique to solve inverse problems for the damped and undamped wave equations in both 1D and 2D spatial dimensions, efficiently approximating the gradients of the objective functional and accurately recovering smooth initial data from boundary measurements.   

The contributions of our work are threefold: 
\begin{itemize}
\item the systematic integration of the DP-SBP operators in space and time with weak imposition of boundary and initial data for the second order wave equation,
\item the derivation of fully discrete energy estimates,
\item the attainment of fully discrete adjoint consistency.
\end{itemize}
To the best of our knowledge, these contributions have never been reported in the literature.

The paper is organized as follows. In Section 2, we introduce the model wave equation, formulate the IBVP, and present the corresponding adjoint problem. In Section 3, we develop semi-discrete and fully discrete numerical schemes for solving the IBVP. For each case, we obtain energy estimates and derive the corresponding adjoint approximations. Numerical experiments are presented in Section 4, and Section 5 draws conclusions.

\section{The continuous problem}
\label{sec:main}

In this section, we state the model wave equation that will be considered throughout this paper, together with well-posed initial and boundary conditions. We then derive energy estimates for the IBVP and formulate the corresponding adjoint problem.

\subsection{Forward problem}

We consider the wave equation
    \begin{equation} \label{contwaveeq}
        \frac{\partial^2 u}{\partial t^2} + \sigma^2 \frac{\partial u}{\partial t} = c^2\frac{\partial^2 u}{\partial x^2} + S, \quad x \in \Omega = [x_L, x_R], \quad t \in [0,T],
    \end{equation}
where $u = u(x,t)$ represents the displacement, $c = c(x) \neq 0$ is the wave propagation speed, $\sigma = \sigma(x)$ is the damping coefficient, and $S = S(x,t)$ is a source term. We augment \eqref{contwaveeq} with the following initial displacement field and particle velocity:
\begin{equation} \label{initcond}
    u(x,0) = f(x) \quad \text{and} \quad \left.\frac{\partial u}{\partial t}\right|_{t = 0} = g(x),
\end{equation}
where $f \in H^1{(\Omega)}$ and $g \in L^2{(\Omega)}$ are sufficiently smooth and compactly supported functions in $\Omega$. {We use the notation $(\cdot,\cdot)_\Omega$ and $\|\cdot\|_{L^2(\Omega)}$ for the standard $L^2$ inner product and norm on $\Omega$, respectively.}

To ensure well-posedness, appropriate boundary conditions must be introduced for \eqref{contwaveeq} on the boundary $\partial \Omega = \{x_L, x_R\}$. In this paper, we focus on linear and energy stable boundary conditions, where the associated energy is given by
\begin{equation} \label{contenergy}
    E(t) = \frac{1}{2}\left( \left\| \frac{1}{c} \frac{\partial u}{\partial t}\right\|^2_{L^2(\Omega)} + \left\|  \frac{\partial u}{\partial x}\right\|^2_{L^2(\Omega)} \right).
\end{equation}
{Following \cite{mattsson2012summation}}, we consider the boundary conditions
\begin{equation} \label{wavebcs}
    \alpha_L \left.\frac{\partial u}{\partial x} \right|_{x=x_L} + \beta_L \left.\frac{\partial u}{\partial t} \right|_{x=x_L} = b_L(t) \quad \text{and} \quad \alpha_R \left.\frac{\partial u}{\partial x}\right|_{x=x_R} + \beta_R \left. \frac{\partial u}{\partial t}\right|_{x=x_R} = b_R(t),
\end{equation}
where $\alpha_L, \beta_L, \alpha_R, \beta_R$ are {real-valued} boundary parameters satisfying
\begin{equation} \label{constantscond}
\alpha_L^2 + \beta_L^2 \neq 0, \quad \alpha_R^2 + \beta_R^2 \neq 0, \quad \alpha_L \beta_L \leq 0, \quad \alpha_R\beta_R \geq 0.
\end{equation}
and $b_L$, $b_R$ are given left and right boundary data, respectively. For the energy analysis, we introduce the boundary terms
\begin{equation}\label{eq:BT}
\begin{split}
& \operatorname{BT}_L(t)  = \left \{
\begin{array}{rl}
 -\frac{\beta_L}{\alpha_L} |\frac{\partial u(x_L, t)}{\partial t}|^2,  & \text{if} \quad \alpha_L\ne 0,\\
0,
 & \text{if} \quad \alpha_L = 0,
\end{array} \right.
\end{split}
\quad
\begin{split}
& \operatorname{BT}_R(t)  = \left \{
\begin{array}{rl}
 \frac{\beta_R}{\alpha_R} |\frac{\partial u(x_R, t)}{\partial t}|^2,  & \text{if} \quad \alpha_R\ne 0,\\
0,
 & \text{if} \quad \alpha_R = 0,
\end{array} \right.
\end{split}
\end{equation}
which are nonnegative under the conditions \eqref{constantscond}. We establish the energy stability of the wave equation IBVP in the following theorem.

\begin{theorem}\label{theo:continuous_stability}
Consider the IBVP \eqref{contwaveeq}, \eqref{initcond},  \eqref{wavebcs} with $\sigma^2 \ge 0$ and homogeneous boundary data $b_L = b_R = 0$. If the boundary parameters $\alpha_L, \beta_L, \alpha_R, \beta_R$ satisfy the conditions \eqref{constantscond},
then in a source-free medium $(S = 0)$, the energy \eqref{contenergy} satisfies the equation
\begin{equation} \label{contestimate_no_source}
    E(t) + \int_{0}^t\left(\left\| \left.\frac{\sigma}{c} \frac{\partial u}{\partial t}\right|_{t = s} \right\|^2_{L^2(\Omega)} + \operatorname{BT}(s)\right) ds=  E(0), \quad \forall t \in [0,T],
\end{equation}
where $\operatorname{BT}(t) = \operatorname{BT}_L(t) + \operatorname{BT}_R(t) {\geq 0}$. %
\end{theorem}
\begin{proof}
Multiplying \eqref{contwaveeq} by $\frac{1}{c^2} \frac{\partial u}{\partial t}$, integrating over the spatial domain $\Omega$, and applying integration by parts, we obtain
\begin{equation}\label{eq:energy_equation_stability}
    \frac{1}{2} \frac{d}{dt} \left( \left\| \frac{1}{c}\frac{\partial u}{\partial t}\right\|^2_{L^2(\Omega)} + \left\| \frac{\partial u}{\partial x}\right\|^2_{L^2(\Omega)}\right)+  \left\|\frac{\sigma}{c} \frac{\partial u}{\partial t}\right\|^2_{L^2(\Omega)}  = \left.\left(\frac{\partial u}{\partial t}\frac{\partial u}{\partial x} \right) \right|_{x=x_L}^{x=x_R} + \left(S, \frac{1}{c^2}\frac{\partial u}{\partial t} \right)_\Omega.
\end{equation}
We set $S=0$ and impose the boundary conditions \eqref{wavebcs} with homogeneous boundary data $b_L = b_R = 0$, yielding 
\begin{equation}\label{eq:energy_equation_proof}
    \frac{d}{dt} E(t) +  \left\|\frac{\sigma}{c} \frac{\partial u}{\partial t}\right\|^2_{L^2(\Omega)} + \operatorname{BT}_L(t) + \operatorname{BT}_R(t) = 0.
\end{equation}
Integrating \eqref{eq:energy_equation_proof} in time completes the proof of the theorem.
\end{proof}

A numerical method that satisfies a discrete analogue of Theorem \ref{theo:continuous_stability} is called energy stable. When the damping coefficient does not vanish, that is $\sigma(x) \ne 0$ for all $x \in \Omega$, we can bound the energy in terms of the initial energy and the source term. A more precise statement is given by the corollary below{, which follows by applying Young's inequality to the source term in \eqref{eq:energy_equation_stability}.}
\begin{corollary} \label{contcorollary}
    Consider the IBVP \eqref{contwaveeq}, \eqref{initcond},  \eqref{wavebcs} with $\sigma^2 > 0$ and homogeneous boundary data $b_L = b_R = 0$. If the boundary parameters $\alpha_L, \beta_L, \alpha_R, \beta_R$ satisfy the conditions \eqref{constantscond},
then the energy \eqref{contenergy} satisfies the estimate
\begin{equation} \label{contestimate_source}
    E(t) + \int_{0}^t \operatorname{BT}(s) \: ds \le   E(0) + \frac{1}{\delta}\int_{0}^t\left\|\frac{S(\cdot, s)}{c}\right\|_{L^2(\Omega)}^2 \: ds, \quad \forall t \in [0,T],
\end{equation}
where $\delta = 2 \min_{x \in \Omega} (\sigma(x))^2$ and $\operatorname{BT}(t) = \operatorname{BT}_L(t) + \operatorname{BT}_R(t)$.
\end{corollary}

We also observe that the square root of the energy \eqref{contenergy} is a quasi $H^1$ norm. Thus, if we consider {$\sqrt{{E}(t)}\ge 0$}, we can obtain a more general result when $\sigma^2 \ge 0$ and $S\ne 0$.
\begin{theorem} \label{theo:continuous_stability2}
Consider the IBVP \eqref{contwaveeq}, \eqref{initcond},  \eqref{wavebcs} with $\sigma^2 \ge 0$ and homogeneous boundary data $b_L = b_R = 0$. If the boundary parameters $\alpha_L, \beta_L, \alpha_R, \beta_R$ satisfy the conditions \eqref{constantscond},
then the square root of the  energy \eqref{contenergy} satisfies the estimate
\begin{equation} \label{contestimate}
    \sqrt{E(t)} \leq e^{-\gamma t} \left( \sqrt{E(0)} + \int_{0}^t e^{\gamma s} \left\| \frac{1}{c}S(\cdot,s)\right\|_{L^2(\Omega)} ds\right), \quad \forall t \in [0,T],
\end{equation}
where $\gamma  = \min_{s\in[0,t]} \frac{\left\| \frac{\sigma}{c} \left.\frac{\partial u}{\partial t}\right|_{t=s} \right\|^2_{L^2(\Omega)} + \operatorname{BT}_L(s) + \operatorname{BT}_R(s)}{2E(s)}$.
\end{theorem} { \begin{proof} Applying the Cauchy-Schwarz inequality to the source term in \eqref{eq:energy_equation_stability} yields a differential inequality for $\sqrt{E(t)}$. We then use the integrating factor method to complete the proof. \end{proof} }

\subsection{Adjoint problem}

We now derive the adjoint problem of the IBVP \eqref{contwaveeq}, \eqref{initcond},  \eqref{wavebcs}. The associated Lagrangian is \begin{equation}
    \mathcal{L}(u, \lambda) = \int_{0}^T{\Phi(u)} \: dt + \int_{0}^T \left( \lambda,\frac{1}{c^2} \frac{\partial^2 u}{\partial t^2} + \left(\frac{\sigma}{c}\right)^2\frac{\partial u}{\partial t} - \frac{\partial^2 u}{\partial x^2} - \frac{1}{c^2}S \right)_{\Omega} dt,
 \end{equation}
where $\lambda$ is the adjoint variable. We compute the first variation of the Lagrangian with respect to $u$ and use integration by parts. Then, we substitute $\tau = T - t$ and introduce $\bar{\lambda}(x,\tau) = \lambda(x,T-\tau)$. This yields the following adjoint IBVP:
\begin{align} 
\begin{split} \label{contadjointIBVP}
    &\frac{\partial^2 \bar{\lambda}}{\partial^2\tau }  + \sigma^2 \frac{\partial \bar{\lambda}}{\partial \tau} = c^2 \frac{\partial^2 \bar{\lambda}}{\partial x^2} - c^2 {\Phi'(\bar{u})}, \quad x \in \Omega = [x_L, x_R], \quad \tau \in [0,T], \\
    &\bar{\lambda}(x, 0) = 0 \quad \text{and} \quad \left.\frac{\partial \bar{\lambda}}{\partial \tau}\right|_{\tau = 0} = 0, \\
    &\alpha_L \left.\frac{\partial \bar{\lambda}}{\partial x}\right|_{x=x_L} + \beta_L \left.\frac{\partial \bar{\lambda}}{\partial \tau}\right|_{x=x_L} = 0 \quad \text{and} \quad \left.\alpha_R \frac{\partial \bar{\lambda}}{\partial x}\right|_{x=x_R} + \beta_R \left.\frac{\partial \bar{\lambda}}{\partial \tau}\right|_{x=x_R} = 0,
\end{split}
\end{align}
where $\bar{u}(x,\tau) = u(x,T-\tau)$ and ${\Phi'}$ is the Fréchet derivative of ${\Phi}$.

Owing to {the time-reversed adjoint variable} $\bar{\lambda}$, the adjoint IBVP above is forward in time {for $\tau \in [0,T]$}, and takes the form of the IBVP \eqref{contwaveeq}, \eqref{initcond},  \eqref{wavebcs}. Therefore, if the boundary parameters $\alpha_L, \beta_L, \alpha_R, \beta_R$ satisfy the conditions \eqref{constantscond}, the adjoint IBVP satisfies energy estimates similar to the ones satisfied by the forward IBVP. 

To state the energy estimates, we define the energy and boundary terms for the adjoint IBVP as follows:
\begin{equation} \label{contenergyadjoint}
    \widetilde{E}(\tau) = \frac{1}{2}\left(\left\| \frac{1}{c}\frac{\partial\bar{\lambda} }{\partial \tau}\right\|^2_{L^2(\Omega)} + \left\| \frac{\partial \bar{\lambda}}{\partial x}\right\|^2_{L^2(\Omega)} \right)
\end{equation}
and
\begin{equation}\label{eq:adjointBT}
\begin{split}
& \widetilde{\operatorname{BT}}_L(\tau)  = \left \{
\begin{array}{rl}
 -\frac{\beta_L}{\alpha_L} |\frac{\partial \bar{\lambda}(x_L, \tau)}{\partial \tau}|^2,  & \text{if} \quad \alpha_L\ne 0,\\
0,
 & \text{if} \quad \alpha_L = 0,
\end{array} \right.
\end{split}
\quad
\begin{split}
& \widetilde{\operatorname{BT}}_R(\tau)  = \left \{
\begin{array}{rl}
 \frac{\beta_R}{\alpha_R} |\frac{\partial \bar{\lambda}(x_R, \tau)}{\partial \tau}|^2,  & \text{if} \quad \alpha_R\ne 0,\\
0,
 & \text{if} \quad \alpha_R = 0,
\end{array} \right.
\end{split}
\end{equation}
respectively. The energy estimate below follows directly from Corollary \ref{contcorollary}.
\begin{theorem} \label{cont_adjoint_theo1}
    Consider the adjoint IBVP \eqref{contadjointIBVP} with $\sigma^2 > 0$. If the boundary parameters $\alpha_L$, $\beta_L$, $\alpha_R$, $\beta_R$ satisfy the conditions \eqref{constantscond}, then the energy \eqref{contenergyadjoint} satisfies the estimate
    \begin{equation}
        \widetilde{E}(\tau) + \int_{0}^\tau \widetilde{\operatorname{BT}}(s) \: ds \le   \widetilde{E}(0) + \frac{1}{\delta}\int_{0}^\tau \left\|{{\Phi'(\bar{u}(\cdot, s))}}\right\|_{L^2(\Omega)}^2 \: ds, \quad \forall \tau \in [0,T],
    \end{equation}
    where $\delta = 2 \min_{x \in \Omega} (\sigma(x))^2$ and $\widetilde{\operatorname{BT}}(\tau) = \widetilde{\operatorname{BT}}_L(\tau) + \widetilde{\operatorname{BT}}_R(\tau)$.
\end{theorem}
Furthermore, we can consider the square root of the energy \eqref{contenergyadjoint}, which gives the following energy estimate when $\sigma^2 \geq 0$.
\begin{theorem} \label{cont_adjoint_theo2}
    Consider the adjoint IBVP \eqref{contadjointIBVP} with $\sigma^2 \geq 0$. If the boundary parameters $\alpha_L$, $\beta_L$, $\alpha_R$, $\beta_R$ satisfy the conditions \eqref{constantscond}, then the square root of the energy \eqref{contenergyadjoint} satisfies the estimate
    \begin{equation} 
    \sqrt{\widetilde{E}(\tau)} \leq e^{-\gamma \tau} \left( \sqrt{\widetilde{E}(0)} + \int_{0}^\tau e^{\gamma s} \left\| {\Phi'(\bar{u}(\cdot, s))}\right\|_{L^2(\Omega)} \: ds\right), \quad \forall \tau \in [0,T],
\end{equation}
where $\gamma = \min_{s\in[0,\tau]} \frac{\left\| \frac{\sigma}{c} \left.\frac{\partial \bar{\lambda}}{\partial \tau}\right|_{\tau=s} \right\|^2_{L^2(\Omega)} + \widetilde{\operatorname{BT}}_L(s) + \widetilde{\operatorname{BT}}_R(s)}{2\widetilde{E}(s)}$.
    
\end{theorem}

Recall that a numerical approximation is said to be adjoint consistent if its adjoint is a consistent approximation of the continuous adjoint problem {\cite{hartmann2007adjoint,hicken2014dual}}. The fact that we can obtain energy estimates for the adjoint problem similar to those of the forward problem motivates the development of a fully discrete adjoint consistent numerical framework. This consistency leads to improved accuracy in functional evaluations and their numerical gradients, which is essential for reliable inverse reconstructions and sensitivity analyses.

{
The inverse problem considered in this work is to recover the initial displacement $u(x,0) = f(x)$ for the IBVP \eqref{contwaveeq}, \eqref{initcond}, \eqref{wavebcs} from volume or boundary measurements $u^{\mathrm{obs}}:\Omega\times[0,T]\to\mathbb{R}$. The problem is formulated as the minimization of the misfit functional
\begin{equation} \label{cont_functional}
\mathcal{J}(u) = \int_{0}^T \Phi(u) \: dt, \quad \text{where } \Phi(u) = \frac{1}{2} \int_{\Omega}  (u-u^{\mathrm{obs}})^2\rho \: dx+ \frac{1}{2} \int_{\partial\Omega}  (u-u^{\mathrm{obs}})^2\theta \: ds.
\end{equation}
Here, $u$ is subject to the IBVP \eqref{contwaveeq}, \eqref{initcond}, \eqref{wavebcs}, $\rho:\Omega\to[0,\infty)$ and $\theta:\partial\Omega\to[0,\infty)$ are the volume and boundary observation weights, respectively, and $ds$ denotes the arc-length element. For any admissible perturbation $v$, the Fréchet derivatives of $\Phi$ and $\mathcal{J}$ are given by
\begin{equation}
\begin{aligned}
&\Phi'(u(\cdot,t))[v(\cdot,t)]
=
\int_\Omega \rho(x)
\left(u(x,t)-u^{\mathrm{obs}}(x,t)\right)v(x,t) \: dx \\
&+
\theta(x_L)\left(u(x_L,t)-u^{\mathrm{obs}}(x_L,t)\right)v(x_L,t)
+
\theta(x_R)\left(u(x_R,t)-u^{\mathrm{obs}}(x_R,t)\right)v(x_R,t).
\end{aligned}
\end{equation}
and
\begin{equation}
    \mathcal{J}'(u)[v] = \int_0^T \Phi'(u(\cdot,t))[v(\cdot,t)] \: dt,
\end{equation}
respectively.
}
{
\begin{remark}
    For the above initial displacement reconstruction inverse problem, the Fréchet derivative term $\Phi'$ does not admit an $L^2(\Omega)$ representation when $\theta \neq 0$. This does not pose an issue since the energy estimates in Theorems \ref{cont_adjoint_theo1} and \ref{cont_adjoint_theo2} mainly serve as target estimates that we aim to mimic at the semi-discrete and fully discrete levels.
\end{remark}
}

\section{Numerical discretization}

We discretize the domain $\Omega$ using $N$ uniform grid points $x_1, x_2, \ldots, x_{N}$, where $x_1 = x_L$, $x_{N} = x_R$, and the grid size is $\Delta x = (x_R - x_L)/(N-1)$. Given a function $f \in L^2(\Omega)$, the corresponding grid function $\mathbf{f} = [f(x_1), f(x_2), \ldots, f(x_{N})]^{\top} \in \mathbb{R}^{N}$ is the restriction of $f$ on the grid.

We say that a triplet of linear operators \((D_x^-, D_x^+, H_x)\), where
$D_x^\pm: \mathbb{R}^N \to \mathbb{R}^N$ and $H_x: \mathbb{R}^N \to \mathbb{R}^N$,
satisfies the DP-SBP framework if the following assumptions hold \cite{williams2024full}:

\begin{enumerate}[label={(A.\arabic*)}]
    \item The linear operator $H_x$ defines a positive discrete measure
    \begin{align}
    \begin{split}
   &\left(\mathbf{f}, \mathbf{f} \right)_{H_x}=\mathbf{f}^\top{H_x}\mathbf{f} > 0, \, \quad \text{for all } \mathbf{f} \in \mathbb{R}^{N}\setminus\{\mathbf{0}\} , \\
   &\left(\mathbf{\mathbf{1}}, \mathbf{f} \right)_{H_x} =\sum_{i = 1}^{N} h_i  f(x_i) \Delta x \to \int_{x_L}^{x_R} f(x) \: dx \quad \text{weakly as } N \to \infty,
   \end{split}
\end{align}
    for any $f \in L^2(\Omega)$.
    \item The linear operators \(D_x^\pm\) are consistent with the first derivative, that is
\[
    (D_x^\pm \mathbf{f})_i = \left.\frac{d f}{dx} \right|_{x = x_i}, \quad \text{for all } f \in V^p,
\]
where $V^p$ is the space of polynomials of degree at most $p$.
    \item \label{dualSBP} The linear operators $D^{\pm}_x: \mathbb{R}^{N} \to \mathbb{R}^{N} $ together obey 
    \begin{align} 
        \left( D^{+}_x \mathbf{f}, \mathbf{g} \right)_{H_x} + \left( \mathbf{f}, D^{-}_x \mathbf{g} \right)_{H_x} = f_Ng_N - f_1g_1, \quad \text{for all } \mathbf{f}, \mathbf{g} \in \mathbb{R}^{N}.
    \end{align}
    \item The linear operators $D^{\pm}_x: \mathbb{R}^{N} \to \mathbb{R}^{N} $ together obey 
    \begin{align}
         \left( \mathbf{f}, \left(D^{+}_x-D^{-}_x \right)\mathbf{f} \right)_{H_x} \le 0, \quad \text{for all } \mathbf{f} \in \mathbb{R}^{N}.
    \end{align}
\end{enumerate}
We will use the linear operators \(D_x^\pm\) to approximate spatial derivatives at the grid points.
\begin{remark}
The traditional SBP operator pair $(D_x, H_x)$ also satisfies the DP-SBP framework by setting $D^-_x =D^+_x =D_x$, so that $\left( \mathbf{f}, \left(D^{+}_x-D^{-}_x \right)\mathbf{f} \right)_{H_x} = 0$ for all $\mathbf{f} \in \mathbb{R}^{N}$.
 Similarly, given an upwind DP-SBP operator triplet $(D^-_x, D^+_x, H_x)$ that satisfies Assumptions (A.1)-(A.4), the averaged operator
$
{D_x^{\mathrm{o}}} := \frac12 \left(D^+_x + D^-_x\right)
$
is a centered-difference-based SBP operator that satisfies the traditional SBP framework \cite{mattsson2017diagonal}.    
\end{remark}

\subsection{The semi-discrete problem}

In this subsection, we use the DP-SBP operators to derive a semi-discrete (time continuous) numerical approximation of the IBVP \eqref{contwaveeq}, \eqref{initcond}, \eqref{wavebcs}. The boundary conditions are imposed weakly by adding penalty terms, and we derive semi-discrete energy estimates that mimic the continuous energy estimates. Finally, the corresponding semi-discrete adjoint problem is obtained, demonstrating that the semi-discrete approximation is adjoint consistent.

\subsubsection{Forward problem}

To begin, we introduce some notation. The variable coefficients are discretized by evaluating them at the grid points as follows:
$$
\pmb{\sigma} = \operatorname{diag}(\sigma(x_1), \sigma(x_2), \ldots,\sigma(x_{N})) \quad \text{and} \quad \mathbf{c} = \operatorname{diag}(c(x_1), c(x_2), \ldots, c(x_{N})).
$$
We also define $e_1 = [1, 0, 0, \ldots, 0]^{\top} \in \mathbb{R}^{N}$ and $e_{N} = [0, 0, \ldots, 0, 1]^\top \in \mathbb{R}^{N}$. The spatial degrees of freedom are denoted by $\mathbf{u}(t) = [u_1(t), u_2(t), \ldots, u_N(t)]^{\top} \in \mathbb{R}^N  $, and the source term $S = S(x,t)$ is represented by the grid function $\mathbf{S}(t) \in \mathbb{R}^{N}$.

For simplicity, we assume $\alpha_L \neq 0$ and $\alpha_R \neq 0$. Using the DP-SBP operators, a consistent semi-discrete numerical approximation of the IBVP \eqref{contwaveeq}, \eqref{initcond}, \eqref{wavebcs} with a weak implementation of the boundary conditions is
{
\small
\begin{align}
\frac{d^2\mathbf{u}}{dt^2} + \pmb{\sigma}^2\frac{d\mathbf{u}}{dt} 
&= \mathbf{c}^2 D^+_x D^-_x \mathbf{u} + \mathbf{S}(t) %
+ \tau_L \mathbf{c}^2 H_x^{-1} e_1 e_1^{\top} \left( D^-_x \mathbf{u} 
+ \frac{\beta_L}{\alpha_L} \frac{d\mathbf{u}}{dt} - \frac{1}{\alpha_L} \mathbf{b}_L(t) \right) \label{semi-discrete} \\
&\quad + \tau_R \mathbf{c}^2 H_x^{-1} e_{N} e_{N}^{\top} \left( D^-_x \mathbf{u} 
+ \frac{\beta_R}{\alpha_R} \frac{d\mathbf{u}}{dt} - \frac{1}{\alpha_R} \mathbf{b}_R(t) \right), 
 \nonumber \\
\mathbf{u}(0) &= \mathbf{f}, \quad \left.\frac{d\mathbf{u}}{dt}\right|_{t=0} = \mathbf{g},
\label{semi-discreteinit}
\end{align}
}where $\tau_L$, $\tau_R$ are penalty parameters to be determined by requiring stability, and $\mathbf{b}_L(t) = b_L(t)e_1$ and $\mathbf{b}_R(t) = b_R(t)e_{N}$ are the left and right boundary data, respectively. To analyze the stability of this numerical approximation, we introduce the semi-discrete energy  \begin{equation} \label{semi-discreteenergy}
     \mathcal{E}(t) := \frac{1}{2} \left(\left\|\mathbf{c}^{-1} \frac{d\mathbf{u}}{dt} \right\|^2_{H_x} + \left\| D^-_x\mathbf{u}\right\|^2_{H_x}\right)\ge 0,
 \end{equation}
 and the semi-discrete boundary term
\begin{equation}\label{semi-discreteBT}
\operatorname{BT}_h(t) = \frac{\beta_R}{\alpha_R} \left(\frac{d}{dt}u_N\right)^2-\frac{\beta_L}{\alpha_L} \left(\frac{d}{dt}u_1\right)^2  \ge 0,
\end{equation}
 which is non-negative due to the conditions \eqref{constantscond}. The following theorem establishes the energy stability of the semi-discrete numerical approximation.
 \begin{theorem} \label{semidiscretestability1}
 Consider the semi-discrete numerical approximation \eqref{semi-discrete}-\eqref{semi-discreteinit} with the penalty parameters $\tau_L = -\tau_R = 1$ and homogeneous boundary data $\mathbf{b}_L = \mathbf{b}_R = \mathbf{0}$. If the boundary parameters $\alpha_L, \beta_L, \alpha_R, \beta_R$  satisfy the conditions \eqref{constantscond} and $\alpha_L$, $\alpha_R$ are nonzero, then in a source-free medium $({\mathbf{S}(t)} = \mathbf{0})$, the semi-discrete energy \eqref{semi-discreteenergy} satisfies the equation
 \begin{equation}
     \mathcal{E}(t) + \int_{0}^t \left(\left\| \mathbf{c}^{-1} \pmb{\sigma} \left. \frac{d\mathbf{u}}{dt}\right|_{t=s}\right\|^2_{H_x} + \operatorname{BT}_h(s)\right) \: ds = \mathcal{E}(0), \quad \forall t \in [0,T], 
 \end{equation}
 where the semi-discrete boundary term $\operatorname{BT}_h$ is defined in \eqref{semi-discreteBT}. %
 \end{theorem}
 \begin{proof}
     Left multiplying \eqref{semi-discrete} by $\left((\mathbf{c}^{-1})^2\frac{d\mathbf{u}}{dt}\right)^\top H_x$ and using the SBP property \ref{dualSBP}, we obtain
     \begin{align}
     \begin{split} \label{eq:semi-discrete-energy-estimate-equation}
         &\frac{1}{2}\frac{d}{dt}\left(\left\| \mathbf{c}^{-1}\frac{d\mathbf{u}}{dt}\right\|_{H_x}^2 
 + \left\| D^-_x\mathbf{u}\right\|_{H_x}^2\right) + \left\| \mathbf{c}^{-1}\pmb{\sigma} \frac{d\mathbf{u}}{dt}\right\|^2_{H_x} = \\
 &\quad \quad \frac{\beta_L}{\alpha_L}\left(\frac{d}{dt}u_1\right)^2 - \frac{\beta_R}{\alpha_R}\left(\frac{d}{dt}u_{N}\right)^2 - \frac{1}{\alpha_L}\frac{du_1}{dt}\mathbf{b}_L + \frac{1}{\alpha_R}\frac{du_N}{dt}\mathbf{b}_R  + \left({\mathbf{S}(t)}, \left(\mathbf{c}^{-1}\right)^2\frac{d\mathbf{u}}{dt} \right)_{H_x},
     \end{split}
     \end{align}
     where we have used the penalty parameters $\tau_L = 1$ and $\tau_R = -1$. We set ${\mathbf{S}(t)} = \mathbf{0}$ and $\mathbf{b}_L = \mathbf{b}_R = \mathbf{0}$, giving
     \begin{equation}\label{eq:semi-discrete-energy-estimate-proof}
         \frac{1}{2}\frac{d}{dt}\left(\left\| \mathbf{c}^{-1}\frac{d\mathbf{u}}{dt}\right\|_{H_x}^2 
 + \left\| D^-_x\mathbf{u}\right\|_{H_x}^2\right) + \left\| \mathbf{c}^{-1}\pmb{\sigma} \frac{d\mathbf{u}}{dt}\right\|^2_{H_x} + \operatorname{BT}_h(t) = 0.
     \end{equation}
    Time integrating \eqref{eq:semi-discrete-energy-estimate-proof} completes the proof.
 \end{proof}

\begin{remark}
The case where $\alpha_L = 0$ or $\alpha_R = 0$ corresponds to Dirichlet boundary conditions, which can be treated either by strongly enforcing the boundary conditions through injection or applying Nitsche-type weak numerical boundary procedures, see \cite{duru2022conservative}.
\end{remark}
 
Similar to the continuous analysis, the semi-discrete energy can be bounded in terms of the initial energy and the source term when the damping coefficient does not vanish $\pmb{\sigma}^2 \succ 0$. Furthermore, by considering the square root of the semi-discrete energy \eqref{semi-discreteenergy}, a stronger estimate analogous to the continuous case can be obtained. The following corollary and theorem make these statements precise.

\begin{corollary} \label{semidiscretecor}
Consider the semi-discrete numerical approximation \eqref{semi-discrete}-\eqref{semi-discreteinit} with $\pmb{\sigma}^2 \succ 0$, the penalty parameters $\tau_L = -\tau_R = 1$, and homogeneous boundary data $\mathbf{b}_L = \mathbf{b}_R = \mathbf{0}$. If the boundary parameters $\alpha_L, \beta_L, \alpha_R, \beta_R$  satisfy the conditions \eqref{constantscond} and $\alpha_L$, $\alpha_R$ are nonzero, then the semi-discrete energy \eqref{semi-discreteenergy} satisfies the estimate
     \begin{equation}
     \mathcal{E}(t) + \int_{0}^t \operatorname{BT}_h(s) \: ds \leq \mathcal{E}(0) + \frac{1}{\delta} \int_{0}^t \left\|\mathbf{c}^{-1} \mathbf{S}(s) \right\|^2_{H_x} \: ds, \quad \forall t \in [0,T],
 \end{equation}
 where $\delta = 2 \min_{1 \leq i \leq N} (\sigma(x_i))^2$ and the semi-discrete boundary term $\operatorname{BT}_h$ is defined in \eqref{semi-discreteBT}.
\end{corollary}

\begin{theorem} \label{semidiscretestability2}
Consider the semi-discrete numerical approximation \eqref{semi-discrete}-\eqref{semi-discreteinit} with the penalty parameters $\tau_L = -\tau_R = 1$ and homogeneous boundary data $\mathbf{b}_L = \mathbf{b}_R = \mathbf{0}$. If the boundary parameters $\alpha_L, \beta_L, \alpha_R, \beta_R$  satisfy the conditions \eqref{constantscond} and $\alpha_L$, $\alpha_R$ are nonzero, then the square root of the semi-discrete energy \eqref{semi-discreteenergy} satisfies the estimate
  \begin{equation}
     \sqrt{\mathcal{E}(t)} \leq e^{-\gamma_{h}t} \left(\sqrt{\mathcal{E}(0)} + \int_{0}^t e^{\gamma_{h}s} \left\|\mathbf{c}^{-1}\mathbf{S}(s) \right\|_{H_x} \: ds\right), \quad \forall t \in [0,T],
 \end{equation}
where $\gamma_{h} = \min_{s\in[0,t]} \frac{\left\| \mathbf{c}^{-1} \pmb{\sigma} \left. \frac{d\mathbf{u}}{dt}\right|_{t=s}\right\|_{H_x}^2 + \operatorname{BT}_h(s)}{2\mathcal{E}(s)}\ge 0$.
\end{theorem}{
\begin{proof} Similar to the proof of Theorem \ref{theo:continuous_stability2}, applying the Cauchy-Schwarz inequality to the source term in the right hand side of \eqref{eq:semi-discrete-energy-estimate-equation} yields a differential inequality for $\sqrt{\mathcal{E}(t)}$. We then use the integrating factor method to complete the proof. \end{proof}}

{
\begin{remark}
    We note that $\gamma_h$ does not necessarily admit a strictly positive lower bound independent of the grid parameters. For example, $\gamma_h=0$ when $\mathbf{g}=\mathbf{0}$ and $\mathcal{E}(0)>0$. The existence of related positive lower bounds is discussed in \cite{nordstrom2008error, nordstrom2018long} in the context of long-time error estimates.
\end{remark}
}
 
Using the stable penalty parameters $\tau_L = -\tau_R = 1$, we can now rewrite the semi-discrete approximation \eqref{semi-discrete}-\eqref{semi-discreteinit} in a more convenient form. To this end, we introduce the discrete operator
\begin{align}\label{eq:selfadjoint}
    \widetilde{D}_{xx}  = -H_x^{-1}(D_x^{-})^\top H_x D_x^{-}.
\end{align}
The semi-discrete approximation \eqref{semi-discrete}-\eqref{semi-discreteinit} then becomes
{
\begin{align}
\frac{d^2\mathbf{u}}{dt^2} + (\pmb{\sigma}^2 + {B})\frac{d\mathbf{u}}{dt} 
&= \mathbf{c}^2 \widetilde{D}_{xx} \mathbf{u} + \widetilde{\mathbf{S}}(t), \quad 
\left.\mathbf{u}\right|_{t=0} = \mathbf{f}, \quad \left.\frac{d\mathbf{u}}{dt}\right|_{t=0} = \mathbf{g},
\label{semi-discreteinit_simple}
\end{align}
}%
where {
\small 
$ 
{B} = H_x^{-1}\mathbf{c}^2\left(\dfrac{\beta_R}{\alpha_R}e_Ne_N^\top-\dfrac{\beta_L}{\alpha_L}e_1e_1^\top\right)$ and $
\widetilde{\mathbf{S}}(t) = {\mathbf{S}}(t) +   H_x^{-1} \mathbf{c}^2\left(\frac{1}{\alpha_L} \mathbf{b}_L(t) - \frac{1}{\alpha_R} \mathbf{b}_R(t) \right).
$
}

The system of ODEs \eqref{semi-discreteinit_simple} can be integrated in time using any stable ODE solver, such {the classical explicit} Runge-Kutta schemes. However, such approaches do not in general guarantee that the resulting fully discrete problem is adjoint consistent \cite{WaltherandGriewank1999,Walther2007}. This may lead to significant errors in computing the gradients of functionals and loss functions used in inverse problems. To address this issue, one of the main contributions of the present work is the development of a DP-SBP in time framework to solve the system of ODEs \eqref{semi-discreteinit_simple}. This will allow us to prove energy stability and adjoint consistency of the fully discrete problem.

\subsubsection{Adjoint problem}

Let us now obtain the adjoint of the semi-discrete approximation. We begin with the following lemma, which shows that the spatial derivative operator $\widetilde{D}_{xx}$ in the semi-discrete approximation \eqref{semi-discreteinit_simple} is self-adjoint in the spatially discrete setting. 
\begin{lemma} \label{semidiscreteadjointlemma}
    The spatial derivative operator $\widetilde{D}_{xx}$ defined in \eqref{eq:selfadjoint} is self-adjoint with respect to the discrete inner product $(\cdot, \cdot)_{H_x}$, that is
    $$
    (\mathbf{u}, \widetilde{D}_{xx}\mathbf{v})_{H_x} = (\widetilde{D}_{xx}\mathbf{u}, \mathbf{v})_{H_x},
    $$
    for all grid functions $\mathbf{u}, \mathbf{v} \in \mathbb{R}^N$.
\end{lemma}
\begin{proof}
    The proof straightforwardly follows from Assumption A.3.
\end{proof}
Using the semi-discrete approximation \eqref{semi-discreteinit_simple}, the corresponding semi-discrete Lagrangian is
\begin{align}
\begin{split}
    &\pmb{\mathcal{L}}(\mathbf{u}, \pmb{\lambda}) = \int_{0}^T {\pmb{\Phi}(\mathbf{u)}} \: dt\\
    &\quad+ \int_{0}^T \left(\pmb{\lambda},  (\mathbf{c}^{-1})^2 \frac{d^2 \mathbf{u}}{dt^2} + (\mathbf{c}^{-1})^2(\pmb{\sigma}^2 + B) \frac{d\mathbf{u}}{dt} - \widetilde{D}_{xx}\mathbf{u} - (\mathbf{c}^{-1})^2\widetilde{\mathbf{S}}(t)\right)_{H_x} dt,
\end{split}
\end{align}
where ${\pmb{\Phi}}$ is a consistent semi-discrete approximation of ${\Phi}$ and $\pmb{\lambda}$ is the adjoint variable. We compute the first variation of the Lagrangian with respect to $\mathbf{u}$ and apply Lemma \ref{semidiscreteadjointlemma}. Then, we substitute $\tau = T - t$ and introduce the time-reversed adjoint variable $\bar{\pmb{\lambda}}(\tau) = \pmb{\lambda}(T - \tau)$. It follows that when the stable penalty parameters  $\tau_L = -\tau_R = 1$ are used, we obtain the semi-discrete adjoint approximation
\begin{equation} \label{semidiscreteadjoint}
    \frac{d^2\bar{\pmb{\lambda}}}{d\tau^2} + (\pmb{\sigma}^2 + {B})\frac{d\bar{\pmb{\lambda}}}{d\tau} 
= \mathbf{c}^2 \widetilde{D}_{xx} \bar{\pmb{\lambda}} -\mathbf{c}^2{\pmb{\Phi}'(\bar{\mathbf{u}})}, \quad 
\left.\bar{\pmb{\lambda}}\right|_{\tau=0} = \mathbf{0}, \quad \left.\frac{d\bar{\pmb{\lambda}}}{d\tau}\right|_{\tau=0} = \mathbf{0},
\end{equation}
where $\bar{\mathbf{u}}(\tau) = \mathbf{u}(T - \tau)$ is the time-reversed grid function, and the spatial derivative operator $\widetilde{D}_{xx}$ is defined in \eqref{eq:selfadjoint}. {Here, $\pmb{\Phi}'$ denotes the Riesz representation of the Fréchet derivative of $\pmb{\Phi}$ with respect to the $H_x$ inner product.}

We observe that the semi-discrete adjoint approximation \eqref{semidiscreteadjoint} is in the form of the semi-discrete approximation \eqref{semi-discreteinit_simple}. Hence, the semi-discrete adjoint approximation satisfies energy estimates similar to the ones satisfied by the forward semi-discrete approximation.

{For completeness, we present the energy estimates for the adjoint approximation. To this end, we define the corresponding semi-discrete energy and boundary term as follows:}
\begin{equation} \label{semidiscreteadjointenergy}
    \widetilde{\mathcal{E}}(\tau) = \frac{1}{2} \left(\left\|\mathbf{c}^{-1} \frac{d\bar{\pmb{\lambda}}}{d\tau} \right\|^2_{H_x} + \left\| D^-_x\bar{\pmb{\lambda}}\right\|^2_{H_x}\right),
\end{equation}
and
\begin{equation} \label{adjointsemi-discreteBT}
    \widetilde{\operatorname{BT}}_h(\tau) = \frac{\beta_R}{\alpha_R}\left( \frac{d}{d\tau} \bar{\lambda}_N\right)^2 - \frac{\beta_L}{\alpha_L}\left( \frac{d}{d\tau}\bar{\lambda}_1\right)^2 \ge 0,
\end{equation}
respectively. The theorems below can be immediately deduced from Corollary \ref{semidiscretecor} and Theorem \ref{semidiscretestability2}.

\begin{theorem}
    Consider the semi-discrete adjoint approximation \eqref{semidiscreteadjoint} with $\pmb{\sigma}^2 \allowbreak \succ 0$. If the boundary parameters $\alpha_L, \beta_L, \alpha_R, \beta_R$ satisfy the conditions \eqref{constantscond} and $\alpha_L$, $\alpha_R$ are nonzero, then the semi-discrete energy \eqref{semidiscreteadjointenergy} satisfies the estimate
    \begin{equation}
     \widetilde{\mathcal{E}}(\tau) + \int_{0}^\tau \widetilde{\operatorname{BT}}_h(s) \: ds \leq \widetilde{\mathcal{E}}(0) + \frac{1}{\delta} \int_{0}^\tau \left\|{\pmb{\Phi}'(\bar{\mathbf{u}}(s)}) \right\|^2_{H_x} \: ds, \quad \forall \tau \in [0,T],
 \end{equation}
 where $\delta = 2 \min_{1 \leq i \leq N} (\sigma(x_i))^2$ and the semi-discrete boundary term $\widetilde{\operatorname{BT}}_h$ is defined in \eqref{adjointsemi-discreteBT}.
\end{theorem}

\begin{theorem}
    Consider the semi-discrete adjoint approximation \eqref{semidiscreteadjoint}. If the boundary parameters $\alpha_L, \beta_L, \alpha_R, \beta_R$ satisfy the conditions \eqref{constantscond} and $\alpha_L$, $\alpha_R$ are nonzero, then the square root of the semi-discrete energy \eqref{semidiscreteadjointenergy} satisfies the estimate
  \begin{equation}
     \sqrt{\widetilde{\mathcal{E}}(\tau)} \leq e^{-\gamma_{h}\tau} \left(\sqrt{\widetilde{\mathcal{E}}(0)} + \int_{0}^\tau e^{\gamma_{h}s} \left\|{\pmb{\Phi}'(\bar{\mathbf{u}}(s))}\right\|_{H_x} \: ds\right), \quad \forall \tau \in [0,T],
 \end{equation}
 where $\gamma_{h} = \min_{s\in[0,\tau]} \frac{\left\| \mathbf{c}^{-1} \pmb{\sigma} \left. \frac{d\bar{\pmb{\lambda}}}{d\tau}\right|_{\tau=s}\right\|_{H_x}^2 + \widetilde{\operatorname{BT}}_h(s)}{2\widetilde{\mathcal{E}}(s)}\ge 0$.
\end{theorem}

{
Assuming that $\pmb{\Phi}$ and $\pmb{\Phi}'$ are consistent approximations of $\Phi$ and $\Phi'$, respectively, the semi-discrete approximation \eqref{semi-discreteinit_simple} is adjoint consistent since it yields the adjoint approximation \eqref{semidiscreteadjoint}, which is consistent with the continuous adjoint problem \eqref{contadjointIBVP}. For the misfit functional \eqref{cont_functional}, the corresponding semi-discrete approximations of $\Phi$ and $\Phi'$ are
\begin{equation} \label{semidiscrete_phi}
\pmb{\Phi}(\mathbf{u}) = \frac{1}{2}(\mathbf{u}-\mathbf{u}^{\mathrm{obs}})^\top\left(H_x\pmb{\rho}+\pmb{\theta}\right)(\mathbf{u}-\mathbf{u}^{\mathrm{obs}})
\end{equation}
and
\begin{equation} \label{semidiscrete_phi_derivative}
\pmb{\Phi}'(\mathbf{u}) = \left(\pmb{\rho}+H_x^{-1}\pmb{\theta}\right)(\mathbf{u}-\mathbf{u}^{\mathrm{obs}}),
\end{equation}
respectively, where $\pmb{\rho}=\operatorname{diag}\left(\rho(x_1),\ldots,\rho(x_N)\right)$ and $\pmb{\theta}=\operatorname{diag}\left(\theta(x_1),0,\ldots,0,\theta(x_N)\right)$ are the volume and boundary observation weight matrices, respectively, and $\mathbf{u}^{\mathrm{obs}}\in\mathbb{R}^N$ is the semi-discrete measurement vector. Clearly, $\pmb{\Phi}$ is a consistent approximation of $\Phi$, while the consistency of $\pmb{\Phi}'$ follows from
\begin{equation}
\begin{aligned}
(\pmb{\Phi}'(\mathbf{u}),\mathbf{v})_{H_x} &= \left(\pmb{\rho}(\mathbf{u}-\mathbf{u}^{\mathrm{obs}}),\mathbf{v}\right)_{H_x} \\
&\quad + \theta(x_1)(u_1-u_1^{\mathrm{obs}})v_1+\theta(x_N)(u_N-u_N^{\mathrm{obs}})v_N,
\end{aligned}
\end{equation}
for any grid function $\mathbf{v} \in \mathbb{R}^N$.}

\subsection{The fully discrete problem}

In this subsection, we construct a fully discrete numerical approximation of the IBVP \eqref{contwaveeq}, \eqref{initcond}, \eqref{wavebcs} by utilizing the semi-discrete approximation \eqref{semi-discreteinit_simple} and extending the DP-SBP framework to time. As in the previous section, we weakly impose the initial conditions by adding penalty terms. Fully discrete energy estimates are derived to demonstrate the stability of the approximation. Finally, we obtain the corresponding adjoint problem, showing that the fully discrete approximation is adjoint consistent.

\subsubsection{Forward problem}

We discretize the time interval $[0, T]$ using $M$ uniform grid points $t_1, t_2, \ldots, t_{M}$, where $t_1 = 0$, $t_{M} = T$, and the time-step size is $\Delta t = T/(M-1)$.
 The space-time degrees of freedom are denoted by the stacked vector  
\[
\mathbf{U} = \begin{bmatrix}
\mathbf{U}_1^\top \: \:  \mathbf{U}_2^\top \: \: \cdots \: \:\mathbf{U}_M^\top
\end{bmatrix}^\top \in \mathbb{R}^{MN},
\]
where \( \mathbf{U}_k = [U_{k1}, U_{k2}, \ldots, U_{kN}]^\top \in \mathbb{R}^N \) contains the spatial degrees of freedom at time {$t_{k}=(k-1) \Delta{t}$ for $k = 1,2, \ldots, M$}. {The fully discrete source terms $\mathbf{S} \in \mathbb{R}^{MN}$ and $\widetilde{\mathbf{S}} \in \mathbb{R}^{MN}$ are obtained by evaluating $\mathbf{S}(t)$ and $\widetilde{\mathbf{S}}(t)$ at the time grid points, respectively.} Similar to the previous section, we also define $e_1 = [1, 0, 0, \ldots, 0]^{\top} \in \mathbb{R}^{M}$ and $e_{M} = [0, 0, \ldots, 0, 1]^\top \in \mathbb{R}^{M}$, which are vectors that can be used to {restrict} the degrees of freedom to the end points $t_1= 0$ and $t_M=T$ of the time interval $[0,T]$. 

{Analogous to the spatial discretization in the semi-discrete case, the time discretization is based on the DP-SBP framework. We introduce the symmetric positive definite quadrature matrix $H_t\in\mathbb{R}^{M\times M}$ and the linear time operators $D_t^j:\mathbb{R}^M\to\mathbb{R}^M$ for $j\in\{-,\mathrm{o},+\}$. Here, $D_t^\mathrm{o}=\frac12(D_t^-+D_t^+)$ denotes the averaged operator, and each $D_t^j$ is a consistent approximation of the first derivative in time.} The operator pair ${(H_t, D^{j}_t)}$ must satisfy Assumptions A.1-A.2 above and the following upwind assumption:
\begin{enumerate}[label={(B.\arabic*)}]
     \item \label{dualSBP_time} The symmetric part of the linear operator
     $
    {Q^{j}_t} = {H_tD^{j}_t} -\frac{1}{2}\left(e_Me_M^\top -  e_1e_1^\top\right),
    $
 that is ${S^{j}_t:= Q^{j}_t +  (Q^{j}_t)^\top}$, satisfies
    \begin{align} \label{timeupwindassumption}
          \mathbf{f}^\top {S^{j}_t} \mathbf{f}  \ge 0,
    \end{align}
    for all $\mathbf{f} \in \mathbb{R}^M$ and each $j = 1,2$.
\end{enumerate}
Admissible time derivative operators {include $D^{-}_t$ and $D^\mathrm{o}_t$.} For the centered difference based SBP operator {$D_t^\mathrm{o}$} we have $\mathbf{f}^\top S^{\mathrm{o}}_t \mathbf{f}  = 0$, and for the upwind DP-SBP operator $D^-_t$ we have $\mathbf{f}^\top S^{-}_t \mathbf{f} \ge 0$, see \cite{mattsson2017diagonal}. The corresponding second derivative operators are
$D_{tt} \in {\{D^{-}_tD^{-}_t, D^{-}_tD^{\mathrm{o}}_t, D_t^\mathrm{o}D^{-}_t, D_t^\mathrm{o}D_t^\mathrm{o}\}}$, all of which are consistent approximations of the continuous second derivative $d^2/dt^2$ on the {time} grid.

The fully discrete approximation is
\begin{equation} \label{fullydiscrete}
    (D_{tt} \otimes I_x)\mathbf{U} + ({D^{j}_t} \otimes (\pmb{\sigma}^2 + {B}))\mathbf{U} - (I_t \otimes \mathbf{c}^2\widetilde{D}_{xx})\mathbf{U} = \operatorname{SAT} + {\widetilde{\mathbf{S}}},
\end{equation}
where $D_{tt} = {D^i_t D^j_t}$ for {$i,j \in \{-,\mathrm{o}\}$}, $I_x$ and $I_t$ are identity operators in space and time, respectively. {The initial conditions are weakly implemented through the penalty terms}
\begin{align}\label{eq:SAT}
    \operatorname{SAT} = &\mu_1(H_t^{-1} e_1e_1^{\top}  \otimes (\pmb{\sigma}^2 + {B}))(\mathbf{U} - \mathbf{F}) +  \mu_2({D^{i}_t}H_t^{-1}e_1e_1^{\top} \otimes I_x)(\mathbf{U} - \mathbf{F})
    \\ \nonumber
    &+ \mu_3 (H_t^{-1} e_1e_1^{\top} \otimes I_x)\left(({D_t^{j}} \otimes I_x)\mathbf{U} - \mathbf{G}\right) \\
    \nonumber
    &+ \mu_4(H_t^{-1}e_1e_1^{\top}H_t^{-1}e_1e_1^\top  \otimes I_x)(\mathbf{U} - \mathbf{F}).
\end{align}
Here, the initial data vectors $\mathbf{F} = (e_1 \otimes I_x)\mathbf{f}$ and $\mathbf{G} = (e_1 \otimes I_x) \mathbf{g}$ are given by padding the previous initial data vectors $\mathbf{f}$ and $\mathbf{g}$ with zeros to ensure consistent matrix dimensions. As will be shown later, to guarantee stability, we will choose the penalty parameters for the penalty terms such that  \begin{align}\label{eq:init_param}
    \mu_1 < -1/2, \quad \mu_3 < -1/2, \quad \mu_2 = \mu_1, \quad \mu_4 = -\mu_1\mu_3.
\end{align}

To simplify the analysis, we introduce the auxiliary field $\mathbf{V}$ and rewrite the fully discrete approximation \eqref{fullydiscrete} as 
{\small
\begin{align} 
\mathbf{V} &= ({D^{j}_t} \otimes I_x)\mathbf{U} - \mu_1 (H_t^{-1}e_1e_1^\top \otimes I_x)( \mathbf{U} - \mathbf{F}),
\label{auxdiscrete}
\\ 
({D^{i}_t} \otimes I_x)\mathbf{V} + \left(I_t \otimes (\pmb{\sigma}^2 + B) \right)\mathbf{V} &= (I_t \otimes \mathbf{c}^2\widetilde{D}_{xx})\mathbf{U} + \mu_3 (H_t^{-1}e_1e_1^{\top} \otimes I_x) (\mathbf{V} - \mathbf{G}) + {\widetilde{\mathbf{S}}},
\label{auxdiscrete2}
\end{align}}%
where we have used the parameter choices $\mu_2 = \mu_1$ and $\mu_4 = -\mu_1 \mu_3$. Then, we define the fully discrete energy
\begin{equation} \label{fullydiscreteenergy}
\mathcal{E}_{k} = \left \{
\begin{array}{rl}
 \frac{1}{2}\left( \left\| \mathbf{c}^{-1}\mathbf{V}_k\right\|_{H_x}^2 + \left\|D^-_x\mathbf{U}_k \right\|^2_{H_x}\right),  & \text{if} \quad k = 1,2,\ldots,M,\\
\frac{1}{2}\left( \left\| \mathbf{c}^{-1}\mathbf{g}\right\|_{H_x}^2 + \left\|D^-_x\mathbf{f} \right\|^2_{H_x}\right),
 & \text{if} \quad k = 0.
\end{array} \right.
\end{equation}
Here, $\mathcal{E}_0$ is the discrete energy of the initial data. To begin, we state the following lemma.
\begin{lemma}\label{lem:lemmata}
    Consider the fully discrete approximation \eqref{fullydiscrete} or  \eqref{auxdiscrete}-\eqref{auxdiscrete2} with homogeneous boundary data $\mathbf{b}_L = \mathbf{b}_R = \mathbf{0}$, where ${D^{i}_t}$, ${D^{j}_t}$ are discrete time derivative operators that satisfy Assumptions A.1-A.2 and B.1. 
    If the penalty parameters $\mu_1, \mu_2, \mu_3, \mu_4$ satisfy \eqref{eq:init_param} and  $\mu_1, \mu_3 < -1/2$, then for all $\Delta{t}>0$ we have
    {\small
    \begin{align}\label{eq:lemmata}
\begin{split}
&\mathcal{E}_M +  \sum_{k=1}^M h_k \left(\|\mathbf{c}^{-1}\pmb{\sigma}\mathbf{V}_k \|_{H_x}^2+ \operatorname{BT}_{hk} \right) \Delta t  
\leq C \mathcal{E}_0 + \sum_{k=1}^M h_k \|\mathbf{c}^{-1}\mathbf{V}_k \|_{H_x} \|\mathbf{c}^{-1}\mathbf{S}_k\|_{H_x}\Delta t, 
\end{split}
\end{align}
}%
where
$\operatorname{BT}_{hk} = \frac{\beta_R}{\alpha_R}V_{kN}^2-\frac{\beta_L}{\alpha_L}V_{k1}^2  \ge 0$,
$h_k > 0$ and $C>0$ are non-dimensional constants independent of the space-time grid parameters.
\end{lemma}{\begin{proof} The proof is given in Appendix~\ref{app:proof-stability-lemma}. \end{proof}}
The following two theorems follow from the above lemma.
\begin{theorem} \label{fullydiscretestability1}
    Consider the fully discrete approximation \eqref{fullydiscrete} or  \eqref{auxdiscrete}-\eqref{auxdiscrete2} with homogeneous boundary data $\mathbf{b}_L = \mathbf{b}_R = \mathbf{0}$, where ${D^{i}_t}$, ${D^{j}_t}$ are discrete time derivative operators that satisfy Assumptions A.1-A.2 and B.1. 
    If the penalty parameters $\mu_1, \mu_2, \mu_3, \mu_4$ satisfy \eqref{eq:init_param} and  $\mu_1, \mu_3 < -1/2$ and $\mathbf{S}_k=0$, then for all $\Delta{t}>0$ we have
    {\small
    \begin{align}
\begin{split}
&\mathcal{E}_M +  \sum_{k=1}^M h_k \left(\|\mathbf{c}^{-1}\pmb{\sigma}\mathbf{V}_k \|_{H_x}^2+ \operatorname{BT}_{hk} \right) \Delta t  
\leq C \mathcal{E}_0,
\end{split}
\end{align}
}where
$\operatorname{BT}_{hk} = \frac{\beta_R}{\alpha_R}V_{kN}^2-\frac{\beta_L}{\alpha_L}V_{k1}^2  \ge 0$,
$h_k >0$ and $C>0$ are non-dimensional constants independent of the space-time grid parameters.
\end{theorem}
\begin{proof}
    The proof follows from Lemma \ref{lem:lemmata}. In particular, we set the source term to zero $\mathbf{S}_k=0$ in \eqref{eq:lemmata} to complete the proof.
\end{proof}
\begin{theorem} \label{fullydiscretestability2}
    Consider the fully discrete approximation \eqref{fullydiscrete} or \eqref{auxdiscrete}-\eqref{auxdiscrete2} with homogeneous boundary data $\mathbf{b}_L = \mathbf{b}_R = \mathbf{0}$, where ${D^{i}_t}$, ${D^{j}_t}$ are discrete time derivative operators that satisfy Assumptions A.1-A.2 and B.1. If the penalty parameters $\mu_1, \mu_2, \mu_3, \mu_4$ satisfy \eqref{eq:init_param} and  $\mu_1, \mu_3 < -1/2$ and $\pmb{\sigma}^2 \succ 0$, then for all $\Delta{t}>0$ we have
    {
    \begin{align}
\begin{split}
&\mathcal{E}_M + \sum_{k=1}^M  \operatorname{BT}_{hk}  h_k  \Delta t
\leq C \mathcal{E}_0 + \frac{1}{2\delta}\sum_{k=1}^M h_k \|\mathbf{c}^{-1}{\mathbf{S}}_k \|_{H_x}^2\Delta{t}, 
\end{split}
\end{align}
}where $\operatorname{BT}_{hk} = \frac{\beta_R}{\alpha_R}V_{kN}^2-\frac{\beta_L}{\alpha_L}V_{k1}^2 \ge 0$, $\delta  = 2 \min_{1 \leq i \leq N} (\sigma(x_i))^2$, $h_k >0$ and $C>0$  are non-dimensional constants independent of the space-time grid parameters.
\end{theorem}
\begin{proof}
    Similarly, the proof follows from Lemma \ref{lem:lemmata}. Specifically, to complete the proof we consider \eqref{eq:lemmata} and use the inequality $2\|\mathbf{c}^{-1}\mathbf{V}_k \|_{H_x} \|\mathbf{c}^{-1}{\mathbf{S}}_k\|_{H_x} \le \delta \|\mathbf{c}^{-1}\mathbf{V}_k \|_{H_x}^2 \allowbreak+ \frac{1}{\delta} \|\mathbf{c}^{-1}{\mathbf{S}}_k\|_{H_x}^2$, where $\delta = 2 \min_{1 \leq i \leq N} (\sigma(x_i))^2>0$.
\end{proof}

We focus on the particular choice of stable penalty parameters $\mu_1 = \mu_3 = -1$. Similar to the semi-discrete case, this choice allows us to rewrite the fully discrete approximation in a convenient form. To this end, we introduce the penalized time operators {$\widetilde{D}^{j}_t = D^{j}_t + H_t^{-1}e_1e_1^\top$ for $j \in \{-,\mathrm{o}\}$}. Then, the fully discrete approximation \eqref{fullydiscrete} becomes
\begin{align}
\begin{split} \label{fullydiscrete2}
    ({\widetilde{D}^{i}_t\widetilde{D}^{j}_t} \otimes I_x)\mathbf{U} + ({\widetilde{D}^{j}_t} \otimes (\pmb{\sigma}^2 + {B}))\mathbf{U} - (I_t \otimes \mathbf{c}^2\widetilde{D}_{xx})\mathbf{U} = \widetilde{\operatorname{SAT}} + {\widetilde{\mathbf{S}}},
\end{split}    
\end{align}
where
\begin{equation}\widetilde{\operatorname{SAT}} = \left({\widetilde{D}^{i}_t}H_t^{-1}e_1e_1^\top \otimes I_x \right)\mathbf{F} + \left(H_t^{-1}e_1e_1^\top \otimes (\pmb{\sigma}^2 + B)\right)\mathbf{F} + \left(H_t^{-1}e_1e_1^\top \otimes I_x\right)\mathbf{G}.
\end{equation}
{Here, we recall that the source term $\widetilde{\mathbf{S}}$ includes the boundary data evaluated at the time grid points.} We will use {the formulation \eqref{fullydiscrete2}} to derive the fully discrete adjoint problem and show adjoint consistency.

{
\begin{remark} \label{error_estimates_remark}
Assuming $\Delta t = C\Delta x$ for some constant $C>0$, the error in the fully discrete energy norm is expected to converge at a rate of $O(\Delta x^r)$ with $r = \lfloor p/2 \rfloor - 1/2$, where $p$ is the interior order of accuracy of the second derivative DP-SBP operator, see Appendix~\ref{app:error-estimate}. This suggests that the $L^2$ error converges at a rate of $O(\Delta x^{r+1})$ with $r+1 = \lfloor p/2 \rfloor + 1/2$. Since our goal is to confirm convergence of the fully discrete approximation, an optimal error estimate is left for future work.
\end{remark}
}

\begin{remark}

Since the resulting fully discrete approximation is implicit, it is desirable to obtain a multiblock formulation that allows block-by-block time evolution. This can be easily achieved by extending the convenient form \eqref{fullydiscrete2} to a multiblock in time formulation; see {Appendix~\ref{appendix:multiblock}.}
    
\end{remark}

\subsubsection{Adjoint problem}

We now derive the fully discrete adjoint problem. To begin, we state the adjoints of the penalized {time} operators in the lemma below.
\begin{lemma} \label{fullydiscreteadjointlemma1}
    The penalized time operators
    {
    $\widetilde{D}^{i}_t = D^i_t + H_t^{-1} e_1 e_1^\top$
    and
    $\widetilde{D}^{j}_t = D^j_t + H_t^{-1} e_1 e_1^\top$
    for $i,j \in \{-,\mathrm{o},+\}$}
    satisfy
    \begin{equation}
        H_t{\widetilde{D}^{j}_t}
        =
        -{(\widehat{D}_t^{j})^\top} H_t
        \quad \text{and} \quad
        H_t{\widetilde{D}^{i}_t\widetilde{D}^{j}_t}
        =
        {\left(\widehat{D}^{j}_t\widehat{D}^{i}_t\right)^\top} H_t,
    \end{equation}
    where
    {
    \begin{equation} \label{timeadjointoperator}
        \widehat{D}^{k}_t =
        \begin{cases}
            D^+_t - H_t^{-1}e_Me_M^\top, & \text{if } k = -, \\
            D^\mathrm{o}_t - H_t^{-1}e_Me_M^\top, & \text{if } k = \mathrm{o}, \\
            D^-_t - H_t^{-1}e_Me_M^\top, & \text{if } k = +.
        \end{cases}
    \end{equation}
    }
\end{lemma}{
\begin{proof}
    The proof follows directly from the SBP property \ref{dualSBP}.
\end{proof}}%
We also need the following lemma, which describes time reversal in the discrete setting through the discrete time operators.
\begin{lemma} \label{fullydiscreteadjointlemma2}
    Let $\tau = T - t$ be the time-reversed variable. {The discrete time operators $D^{-}_\tau = -RD^{+}_tR$, $D^{\mathrm{o}}_\tau = -RD^{\mathrm{o}}_tR$, and $D^{+}_\tau = -RD^{-}_tR$ are consistent approximations of $\frac{\partial}{\partial \tau}$}, where $R$ is the backward identity matrix. Furthermore, {for each $k \in \{-,\mathrm{o},+\}$}, the penalized time operator {$\widetilde{D}^{k}_\tau = D_\tau^{k} + H^{-1}_t e_1e_1^\top$} satisfies
\begin{equation}
{\widetilde{D}^{k}_\tau} = - R{\widehat{D}^{k}_t}R,\label{penalizedtauoperators}
\end{equation}
where ${\widehat{D}^{k}_t}$ is the adjoint operator defined in \eqref{timeadjointoperator}.
\end{lemma}
From the fully discrete approximation \eqref{fullydiscrete2}, we obtain the corresponding fully discrete Lagrangian
\begin{align}
\begin{split} \label{fullydiscretelagrangian}
\pmb{\mathbf{L}}(\mathbf{U}, \pmb{\Lambda}) 
&= {\mathbf{J}(\mathbf{U})} \\
&\quad+ \left(\pmb{\Lambda}, 
({\widetilde{D}^{i}_t\widetilde{D}^{j}_t} \otimes  (\mathbf{c}^{-1})^2)\mathbf{U} 
+ ({\widetilde{D}^{j}_t} \otimes  (\mathbf{c}^{-1})^2(\pmb{\sigma}^2 + {B}))\mathbf{U} \right. \\
&\qquad\left. - (I_t \otimes \widetilde{D}_{xx})\mathbf{U} 
- (I_t \otimes (\mathbf{c}^{-1})^2)\left( \widetilde{\operatorname{SAT}} + {\widetilde{\mathbf{S}}} \right) 
\right)_{H_t \otimes H_x},
\end{split}
\end{align}
where ${\mathbf{J}}$ is a consistent fully discrete approximation of $\mathcal{J}$ and $\pmb{\Lambda}$ is the adjoint variable. We compute the first variation of the Lagrangian with respect to $\mathbf{U}$ and use Lemmas \ref{semidiscreteadjointlemma} and \ref{fullydiscreteadjointlemma1}. Then, we substitute $\bar{\pmb{\Lambda}} = (R \otimes I_x)\pmb{\Lambda}$, where $R$ is the backward identity matrix, and apply Lemma \ref{fullydiscreteadjointlemma2}. Hence, when the stable penalty parameters $\mu_1 = \mu_3 = -1$ are used, we have the fully discrete adjoint approximation
\begin{equation} \label{fullydiscreteadjoint}
    \left({\widetilde{D}^{j}_\tau\widetilde{D}^{i}_\tau} \otimes I_x\right)\bar{\pmb{\Lambda}} +\left({\widetilde{D}^{j}_\tau} \otimes (\pmb{\sigma}^2 + {B}) \right)\bar{\pmb{\Lambda}} = (I_t \otimes \mathbf{c}^2\widetilde{D}_{xx})\bar{\pmb{\Lambda}} - \left(R \otimes \mathbf{c}^2\right){\mathbf{J}'(\mathbf{U})},
\end{equation}
where {$i,j \in \{-,\mathrm{o}\}$} and the discrete time operators {$\widetilde{D}^{i}_\tau$, $\widetilde{D}^{j}_\tau$} are defined in \eqref{penalizedtauoperators}. {Here, $\mathbf{J}'$ denotes the Riesz representation of the Fréchet derivative of $\mathbf{J}$ with respect to the $(H_t \otimes H_x)$ inner product.}

We observe that the second time derivative operator in the fully discrete adjoint approximation \eqref{fullydiscreteadjoint} is {$\widetilde{D}^{j}_\tau\widetilde{D}^{i}_\tau$} instead of {$\widetilde{D}^{i}_\tau\widetilde{D}^{j}_\tau$}. Since Lemma \ref{fullydiscreteadjointlemma2} implies that the discrete time operators ${\widetilde{D}^{i}_\tau}$ and ${\widetilde{D}^{j}_\tau}$ satisfy the upwind assumption \eqref{timeupwindassumption}, the fully discrete adjoint approximation satisfies energy stability estimates similar {to} the ones satisfied by the forward fully discrete approximation if ${i = j}$. Therefore, the forward second time derivative operator $D_{tt} = D^i_t D^j_t$ must be either $D^-_t\,D^-_t$ or {$D^\mathrm{o}_t\,D^\mathrm{o}_t$} in order to guarantee energy stability for both the forward and adjoint approximations.

{Similar to the previous section, we present the energy estimates for the adjoint approximation for completeness.} To this end, we define the fully discrete energy for the fully discrete adjoint approximation as follows:
\begin{equation} \label{fullydiscreteadjointenergy}
\widetilde{\mathcal{E}}_{k} = \left \{
\begin{array}{rl}
 \frac{1}{2}\left( \left\| \mathbf{c}^{-1}({\widetilde{D}^{j}_\tau} \bar{\pmb{\Lambda}})_k\right\|_{H_x}^2 + \left\|D^-_x\bar{\pmb{\Lambda}}_k \right\|^2_{H_x}\right),  & \text{if} \quad k = 1,2,\ldots,M,\\
0,
 & \text{if} \quad k = 0.
\end{array} \right.
\end{equation}
The following energy estimate can be immediately deduced from Theorem 
\ref{fullydiscretestability2}.
\begin{theorem}
    Consider the fully discrete adjoint approximation \eqref{fullydiscreteadjoint} with $\pmb{\sigma}^2 \succ 0$, where {$\widetilde{D}^{i}_\tau$, $\widetilde{D}^{j}_\tau$} are discrete time operators defined in \eqref{penalizedtauoperators}. If {$i = j$}, then for all $\Delta t > 0$ we have
    \begin{equation}
    \widetilde{\mathcal{E}}_M +  \sum_{k=1}^M  \widetilde{\operatorname{BT}}_{hk}  h_k  \Delta t \leq  C\widetilde{\mathcal{E}}_0 +\frac{1}{2\delta}\sum_{k=1}^M h_k \left\|\left(\left(R \otimes \mathbf{c}^2\right){\mathbf{J}'(\mathbf{U})}\right)_k\right \|_{H_x}^2\Delta{t},
    \end{equation}
    where $\widetilde{\operatorname{BT}}_{hk} = \frac{\beta_R}{\alpha_R} ({\widetilde{D}^{j}_\tau} \bar{\pmb{\Lambda}})_{kN}^2 - \frac{\beta_L}{\alpha_L} ({\widetilde{D}^{j}_\tau} \bar{\pmb{\Lambda}})_{k1}^2$, $\delta  = 2 \min_{1 \leq i \leq N} (\sigma(x_i))^2$, $h_k >0$ and $C>0$  are non-dimensional constants independent of the space-time grid parameters.
\end{theorem}

{
The fully discrete approximation \eqref{fullydiscrete} or \eqref{auxdiscrete}-\eqref{auxdiscrete2} leads to the fully discrete adjoint approximation \eqref{fullydiscreteadjoint}. Assuming that $\mathbf{J}$ and $\mathbf{J}'$ are consistent approximations of $\mathcal{J}$ and $\mathcal{J}'$, the fully discrete approximation is adjoint consistent since the fully discrete adjoint approximation \eqref{fullydiscreteadjoint} is consistent with the continuous adjoint problem \eqref{contadjointIBVP}. For the misfit functional \eqref{cont_functional},
we extend the semi-discrete approximations
\eqref{semidiscrete_phi}-\eqref{semidiscrete_phi_derivative}
to obtain the fully discrete approximations
\begin{equation}
    \mathbf{J}(\mathbf{U}) = \frac{1}{2}\left(\mathbf{U} - \mathbf{U}^{\mathrm{obs}} \right)^\top\left(H_t \otimes (H_x\pmb{\rho} + \pmb{\theta}) \right)\left(\mathbf{U} - \mathbf{U}^{\mathrm{obs}} \right)
\end{equation}
and
\begin{equation}
    \mathbf{J}'(\mathbf{U}) = \left(I_t \otimes (\pmb{\rho} + H_x^{-1}\pmb{\theta})\right)(\mathbf{U} - \mathbf{U}^{\mathrm{obs}}),
\end{equation}
where $\mathbf{U}^{\mathrm{obs}} \in \mathbb{R}^{MN}$ is the stacked measurements vector. It is straightforward to verify that $\mathbf{J}$ is a consistent approximation of $\mathcal{J}$, while the consistency of $\mathbf{J}'$ follows from $
\left(\mathbf{J}'(\mathbf{U}),\mathbf{V}\right)_{H_t\otimes H_x}
=
(\mathbf{U}-\mathbf{U}^{\mathrm{obs}})^\top
\left(H_t\otimes(H_x\pmb{\rho}+\pmb{\theta})\right)\mathbf{V}
$ for any stacked space-time grid vector $\mathbf{V}\in\mathbb{R}^{MN}$.
}

\section{Numerical experiments}

In this section, we present numerical experiments in one and two spatial dimensions to demonstrate the efficacy and accuracy of the proposed numerical method. We also present numerical simulations of recovering a Gaussian initial displacement field from boundary measurements as a model example to illustrate the potential of the numerical method for solving inverse problems. Forward simulations with the same Gaussian initial displacement field are provided in {Appendix~\ref{appendix:forward_numerical_experiments}.}

\subsection{One-dimensional convergence test} \label{1d_convergence_test}

We consider the IBVP \eqref{contwaveeq}, \eqref{initcond},  \eqref{wavebcs} in the spatial domain $\Omega = [-1,1]$ and the time interval $0 \leq t \leq T = 2$ with zero source term $S = 0$. We set the constant wave speed  $c = 1$ and consider the undamped wave equation $\sigma = 0$ and the damped wave equation $\sigma = 1$. The initial conditions are $u(x,0) = \cos(\pi x)$ and $\left.\frac{\partial u}{\partial t}\right|_{t=0} = 0$, and we use homogeneous Neumann boundary conditions, which are obtained by setting the boundary parameters $\beta_L = \beta_R = 0$ and $\alpha_L = \alpha_R = 1$ with zero boundary data $b_L = b_R = 0$. The exact solutions are given in Appendix~\ref{appendix:exact_solutions}.

We compute the numerical solution using space-time DP-SBP operators of orders $p = 2,4,6,8$ {\cite{mattsson2017diagonal}} with the grid parameters $\Delta x = \Delta t = 0.1, 0.05, 0.025, 0.0125$. We consider the discrete time derivative operators $D_{tt} \in {\{D^{-}_tD^{-}_t, D^{-}_tD^{\mathrm{o}}_t, D_t^\mathrm{o}D^{-}_t, D_t^\mathrm{o}D_t^\mathrm{o}\}}$, and compute the $L^2$ error at the final time $T = 2$ for each case. The observed convergence rates are reported in Table \ref{tab:convergence_side_by_side}, while the error plotted against the spatial steps are shown in Fig. \ref{fig:convergence_dp_sbp_orders} for the undamped wave equation $\sigma = 0$ and Fig. \ref{fig:convergence_dp_sbp_orders2} for the damped wave equation $\sigma = 1$. {Note that the errors converge to zero at rates higher than those predicted by the error estimate discussed in Remark \ref{error_estimates_remark}.}

\begin{table}[htpb]
\centering
\makebox[\textwidth]{%
  \begin{subtable}{0.48\textwidth}
    \centering
    \caption{}
    \resizebox{\textwidth}{!}{%
      \begin{tabular}{|c|cccc|}
        \hline
        \multirow{2}{*}{order} & \multicolumn{4}{c|}{$D_{tt}$} \\ \cline{2-5}
        & \multicolumn{1}{c|}{$D^-_tD^-_t$} & \multicolumn{1}{c|}{$D^-_t{D_t^\mathrm{o}}$} & \multicolumn{1}{c|}{${D_t^{\mathrm{o}}}D^-_t$} & ${D_t^{\mathrm{o}}D_t^{\mathrm{o}}}$ \\ \hline
        2 & \multicolumn{1}{c|}{2.8082} & \multicolumn{1}{c|}{2.7341} & \multicolumn{1}{c|}{2.6454} & 2.4175 \\ \hline
        4 & \multicolumn{1}{c|}{4.4994} & \multicolumn{1}{c|}{4.4350} & \multicolumn{1}{c|}{4.4655} & 4.4735 \\ \hline
        6 & \multicolumn{1}{c|}{4.5039} & \multicolumn{1}{c|}{4.4935} & \multicolumn{1}{c|}{4.4947} & 4.4777 \\ \hline
        8 & \multicolumn{1}{c|}{6.8510} & \multicolumn{1}{c|}{6.8432} & \multicolumn{1}{c|}{6.8519} & 6.8031 \\ \hline
      \end{tabular}
    }
  \end{subtable}
  \hspace{1em}
  \begin{subtable}{0.48\textwidth}
    \centering
    \caption{}
    \resizebox{\textwidth}{!}{%
      \begin{tabular}{|c|cccc|}
        \hline
        \multirow{2}{*}{order} & \multicolumn{4}{c|}{$D_{tt}$} \\ \cline{2-5}
        & \multicolumn{1}{c|}{$D^-_tD^-_t$} & \multicolumn{1}{c|}{$D^-_t{D_t^\mathrm{o}}$} & \multicolumn{1}{c|}{${D_t^{\mathrm{o}}}D^-_t$} & ${D_t^{\mathrm{o}}D_t^{\mathrm{o}}}$ \\ \hline
        2 & \multicolumn{1}{c|}{1.6901} & \multicolumn{1}{c|}{1.8148} & \multicolumn{1}{c|}{1.8403} & 1.9337 \\ \hline
        4 & \multicolumn{1}{c|}{3.7652} & \multicolumn{1}{c|}{3.8150} & \multicolumn{1}{c|}{3.7708} & 3.8334 \\ \hline
        6 & \multicolumn{1}{c|}{4.5145} & \multicolumn{1}{c|}{4.5036} & \multicolumn{1}{c|}{4.5036} & 4.4878 \\ \hline
        8 & \multicolumn{1}{c|}{6.8877} & \multicolumn{1}{c|}{6.8800} & \multicolumn{1}{c|}{6.8829} & 6.8594 \\ \hline
      \end{tabular}
    }
  \end{subtable}
}
\caption{Convergence rates for different second time derivative approximations and operator accuracy orders for (a) the undamped wave equation $\sigma = 0$ and for (b) the damped wave equation $\sigma = 1$.}
\label{tab:convergence_side_by_side}
\end{table}

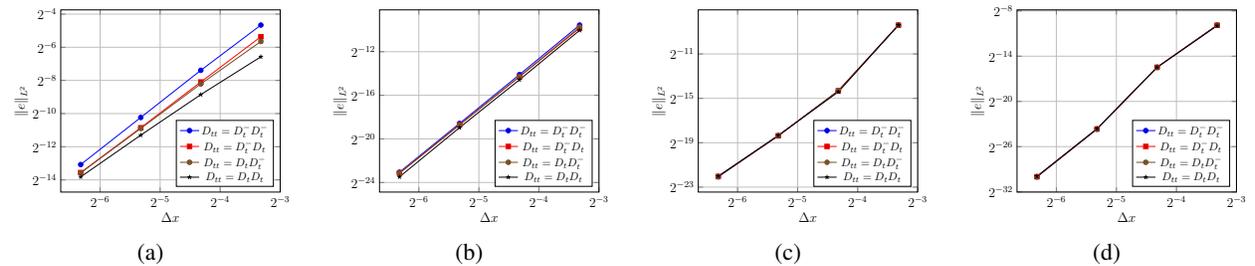
\begin{figure}[htbp]
\centering
\begin{subfigure}{0.48\textwidth}
    \centering
    \resizebox{\textwidth}{!}{
    \begin{tikzpicture}
      \begin{loglogaxis}[
        xlabel={{$\Delta x$}},
        ylabel={{$L^2$ error}},
        label style={font=\large},
        tick label style={font=\large},
        log basis y = {2},
        log basis x = {2},
        legend entries = {
          $D_{tt} = D^-_tD^-_t$,
          $D_{tt} = D^-_t{D_t^{\mathrm{o}}}$,
          $D_{tt} = {D_t^{\mathrm{o}}}D^-_t$,
          $D_{tt} = {D_t^{\mathrm{o}}}{D_t^{\mathrm{o}}}$
        },
        legend pos=south east,
        grid=major,
        xmax = 0.125,
      ]
        \addplot table[x index=0,y index=1,col sep=comma] {error_p2_sigma0.csv};
        \addplot table[x index=0,y index=2,col sep=comma] {error_p2_sigma0.csv};
        \addplot table[x index=0,y index=3,col sep=comma] {error_p2_sigma0.csv};
        \addplot table[x index=0,y index=4,col sep=comma] {error_p2_sigma0.csv};
      \end{loglogaxis}
    \end{tikzpicture}}
    \caption{}
\end{subfigure}
\hfill
\begin{subfigure}{0.48\textwidth}
    \centering
    \resizebox{\textwidth}{!}{
    \begin{tikzpicture}
      \begin{loglogaxis}[
        xlabel={{$\Delta x$}},
        ylabel={{$L^2$ error}},
        label style={font=\large},
        tick label style={font=\large},
        log basis y = {2},
        log basis x = {2},
        legend entries = {
          $D_{tt} = D^-_tD^-_t$,
          $D_{tt} = D^-_t{D_t^{\mathrm{o}}}$,
          $D_{tt} = {D_t^{\mathrm{o}}}D^-_t$,
          $D_{tt} = {D_t^{\mathrm{o}}}{D_t^{\mathrm{o}}}$
        },
        legend pos=south east,
        grid=major,
        xmax = 0.125,
      ]
        \addplot table[x index=0,y index=1,col sep=comma] {error_p4_sigma0.csv};
        \addplot table[x index=0,y index=2,col sep=comma] {error_p4_sigma0.csv};
        \addplot table[x index=0,y index=3,col sep=comma] {error_p4_sigma0.csv};
        \addplot table[x index=0,y index=4,col sep=comma] {error_p4_sigma0.csv};
      \end{loglogaxis}
    \end{tikzpicture}}
    \caption{}
\end{subfigure}

\begin{subfigure}{0.48\textwidth}
    \centering
    \resizebox{\textwidth}{!}{
    \begin{tikzpicture}
      \begin{loglogaxis}[
        xlabel={{$\Delta x$}},
        ylabel={{$L^2$ error}},
        label style={font=\large},
        tick label style={font=\large},
        log basis y = {2},
        log basis x = {2},
        legend entries = {
          $D_{tt} = D^-_tD^-_t$,
          $D_{tt} = D^-_t{D_t^{\mathrm{o}}}$,
          $D_{tt} = {D_t^{\mathrm{o}}}D^-_t$,
          $D_{tt} = {D_t^{\mathrm{o}}}{D_t^{\mathrm{o}}}$
        },
        legend pos=south east,
        grid=major,
        xmax = 0.125,
      ]
        \addplot table[x index=0,y index=1,col sep=comma] {error_p6_sigma0.csv};
        \addplot table[x index=0,y index=2,col sep=comma] {error_p6_sigma0.csv};
        \addplot table[x index=0,y index=3,col sep=comma] {error_p6_sigma0.csv};
        \addplot table[x index=0,y index=4,col sep=comma] {error_p6_sigma0.csv};
      \end{loglogaxis}
    \end{tikzpicture}}
    \caption{}
\end{subfigure}
\hfill
\begin{subfigure}{0.48\textwidth}
    \centering
    \resizebox{\textwidth}{!}{
    \begin{tikzpicture}
      \begin{loglogaxis}[
        xlabel={{$\Delta x$}},
        ylabel={{$L^2$ error}},
        label style={font=\large},
        tick label style={font=\large},
        log basis y = {2},
        log basis x = {2},
        legend entries = {
          $D_{tt} = D^-_tD^-_t$,
          $D_{tt} = D^-_t{D_t^{\mathrm{o}}}$,
          $D_{tt} = {D_t^{\mathrm{o}}}D^-_t$,
          $D_{tt} = {D_t^{\mathrm{o}}}{D_t^{\mathrm{o}}}$
        },
        legend pos=south east,
        grid=major,
        xmax = 0.125,
      ]
        \addplot table[x index=0,y index=1,col sep=comma] {error_p8_sigma0.csv};
        \addplot table[x index=0,y index=2,col sep=comma] {error_p8_sigma0.csv};
        \addplot table[x index=0,y index=3,col sep=comma] {error_p8_sigma0.csv};
        \addplot table[x index=0,y index=4,col sep=comma] {error_p8_sigma0.csv};
      \end{loglogaxis}
    \end{tikzpicture}}
    \caption{}
\end{subfigure}

\caption{{
Log-scale error plots for the undamped wave equation $\sigma = 0$ using the DP-SBP operators of orders (a) 2, (b) 4, (c) 6 and (d) 8 with different second time derivative approximations.}
}
\label{fig:convergence_dp_sbp_orders}
\end{figure}

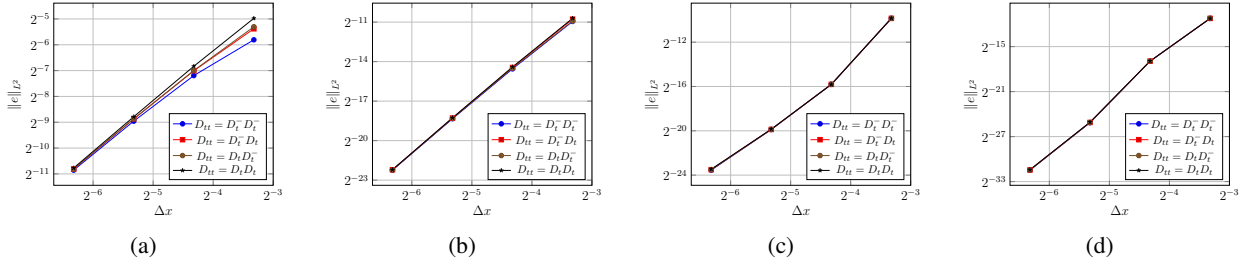
\begin{figure}[htbp]
\centering
\begin{subfigure}{0.48\textwidth}
    \centering
    \resizebox{\textwidth}{!}{
    \begin{tikzpicture}
      \begin{loglogaxis}[
        xlabel={{$\Delta x$}},
        ylabel={{$L^2$ error}},
        label style={font=\large},
        tick label style={font=\large},
        log basis y = {2},
        log basis x = {2},
        legend entries = {
          $D_{tt} = D^-_tD^-_t$,
          $D_{tt} = D^-_t{D_t^{\mathrm{o}}}$,
          $D_{tt} = {D_t^{\mathrm{o}}}D^-_t$,
          $D_{tt} = {D_t^{\mathrm{o}}}{D_t^{\mathrm{o}}}$
        },
        legend pos=south east,
        grid=major,
        xmax = 0.125,
      ]
        \addplot table[x index=0,y index=1,col sep=comma] {error_p2_sigma1.csv};
        \addplot table[x index=0,y index=2,col sep=comma] {error_p2_sigma1.csv};
        \addplot table[x index=0,y index=3,col sep=comma] {error_p2_sigma1.csv};
        \addplot table[x index=0,y index=4,col sep=comma] {error_p2_sigma1.csv};
      \end{loglogaxis}
    \end{tikzpicture}}
    \caption{}
\end{subfigure}
\begin{subfigure}{0.48\textwidth}
    \centering
    \resizebox{\textwidth}{!}{
    \begin{tikzpicture}
      \begin{loglogaxis}[
        xlabel={{$\Delta x$}},
        ylabel={{$L^2$ error}},
        label style={font=\large},
        tick label style={font=\large},
        log basis y = {2},
        log basis x = {2},
        legend entries = {
          $D_{tt} = D^-_tD^-_t$,
          $D_{tt} = D^-_t{D_t^{\mathrm{o}}}$,
          $D_{tt} = {D_t^{\mathrm{o}}}D^-_t$,
          $D_{tt} = {D_t^{\mathrm{o}}}{D_t^{\mathrm{o}}}$
        },
        legend pos=south east,
        grid=major,
        xmax = 0.125,
      ]
        \addplot table[x index=0,y index=1,col sep=comma] {error_p4_sigma1.csv};
        \addplot table[x index=0,y index=2,col sep=comma] {error_p4_sigma1.csv};
        \addplot table[x index=0,y index=3,col sep=comma] {error_p4_sigma1.csv};
        \addplot table[x index=0,y index=4,col sep=comma] {error_p4_sigma1.csv};
      \end{loglogaxis}
    \end{tikzpicture}}
    \caption{}
\end{subfigure}

\begin{subfigure}{0.48\textwidth}
    \centering
    \resizebox{\textwidth}{!}{
    \begin{tikzpicture}
      \begin{loglogaxis}[
        xlabel={{$\Delta x$}},
        ylabel={{$L^2$ error}},
        label style={font=\large},
        tick label style={font=\large},
        log basis y = {2},
        log basis x = {2},
        legend entries = {
          $D_{tt} = D^-_tD^-_t$,
          $D_{tt} = D^-_t{D_t^{\mathrm{o}}}$,
          $D_{tt} = {D_t^{\mathrm{o}}}D^-_t$,
          $D_{tt} = {D_t^{\mathrm{o}}}{D_t^{\mathrm{o}}}$
        },
        legend pos=south east,
        grid=major,
        xmax = 0.125,
      ]
        \addplot table[x index=0,y index=1,col sep=comma] {error_p6_sigma1.csv};
        \addplot table[x index=0,y index=2,col sep=comma] {error_p6_sigma1.csv};
        \addplot table[x index=0,y index=3,col sep=comma] {error_p6_sigma1.csv};
        \addplot table[x index=0,y index=4,col sep=comma] {error_p6_sigma1.csv};
      \end{loglogaxis}
    \end{tikzpicture}}
    \caption{}
\end{subfigure}
\begin{subfigure}{0.48\textwidth}
    \centering
    \resizebox{\textwidth}{!}{
    \begin{tikzpicture}
      \begin{loglogaxis}[
        xlabel={{$\Delta x$}},
        ylabel={{$L^2$ error}},
        label style={font=\large},
        tick label style={font=\large},
        log basis y = {2},
        log basis x = {2},
        legend entries = {
          $D_{tt} = D^-_tD^-_t$,
          $D_{tt} = D^-_t{D_t^{\mathrm{o}}}$,
          $D_{tt} = {D_t^{\mathrm{o}}}D^-_t$,
          $D_{tt} = {D_t^{\mathrm{o}}}{D_t^{\mathrm{o}}}$
        },
        legend pos=south east,
        grid=major,
        xmax = 0.125,
      ]
        \addplot table[x index=0,y index=1,col sep=comma] {error_p8_sigma1.csv};
        \addplot table[x index=0,y index=2,col sep=comma] {error_p8_sigma1.csv};
        \addplot table[x index=0,y index=3,col sep=comma] {error_p8_sigma1.csv};
        \addplot table[x index=0,y index=4,col sep=comma] {error_p8_sigma1.csv};
      \end{loglogaxis}
    \end{tikzpicture}}
    \caption{}
\end{subfigure}

\caption{{
Log-scale error plots for the damped wave equation $\sigma = 1$ using the DP-SBP operators of orders (a) 2, (b) 4, (c) 6 and (d) 8 with different second time derivative approximations.}
}
\label{fig:convergence_dp_sbp_orders2}
\end{figure}

\subsection{One-dimensional Gaussian inverse problem} \label{numerical_inverse_1D}

{
We consider the inverse problem of recovering the initial displacement field described in Section 2 with the observation weights $\rho = 0$ and $\theta = 1$. Hence, the misfit functional \eqref{cont_functional} contains only the boundary term, and the measurements are taken on the boundary. For this experiment, we choose the Gaussian profile $f(x) = e^{-100x^2}$ as the initial displacement field, and the boundary measurements are obtained from the exact solution of the forward IBVP.
}

We compute the numerical solution and the corresponding adjoint using the 4th order space-time DP-SBP operators with the grid parameters $\Delta x = \Delta t = 0.02$. We use the second time derivative operator $D_{tt} = D^-_tD^-_t$. {For comparison, we also consider the classical 4th order Runge-Kutta (RK4) method for time integration with the same grid parameters.} {For both methods}, the gradient is computed using the adjoint state method \cite{plessix2006review}, and we carry out the optimization using the BFGS algorithm \cite{nocedal2006numerical} implemented in \texttt{fminunc} (MATLAB) {for five iterations}; see Appendix~\ref{appendix:gradient} for details. The initial guess of the BFGS iteration is the constant $u=1$ shown in Fig. \ref{fig:init_inverse_gaussian1D_neumann}.

{The absolute error snapshots, the objective functional values, and the $L^2$ errors are shown in Fig. \ref{fig:inverse_gaussian1D_neumann} for the undamped wave equation $\sigma=0$ and Fig. 
\ref{fig:inverse_gaussian1D_neumann2} for the damped wave equation $\sigma=1$. For the undamped case, the DP-SBP objective functional decreases steadily throughout the iterations, whereas the RK4 objective functional levels off after the first few iterations, resulting in a final objective value about two orders of magnitude below that of RK4. For the damped case, the two methods initially exhibit similar decreases, but the DP-SBP method attains lower objective functional values after iteration 2 and a final value about one order of magnitude below that of RK4. The absolute error snapshots show a clear reduction between iterations 2 and 5 in the damped case, while the final $L^2$ errors are of the same order of magnitude for both methods, indicating comparable reconstruction accuracy.}

{
We also perform Taylor remainder tests to assess the accuracy of the discrete gradients computed using both the DP-SBP and RK4 methods. For each method, we compute the Taylor remainder
\begin{equation} \label{taylor_remainder}
R(\varepsilon)
=
\left|
\mathbf{J}\left(\mathbf{U}(\mathbf{f}+\varepsilon\mathbf{p})\right)
-
\mathbf{J}\left(\mathbf{U}(\mathbf{f})\right)
-
\varepsilon\left(\mathbf{d},\mathbf{p}\right)_{H_x}
\right|,
\end{equation}
where $\mathbf{p}$ is a randomly chosen perturbation direction and $\mathbf{d}$ denotes the corresponding discrete gradient. As depicted in Table \ref{tab:taylor_test_1D}, the Taylor remainder exhibits second order convergence for the DP-SBP method and first order convergence for the RK4 method for the undamped wave equation $\sigma = 0$. Similar convergence rates are observed for the damped wave equation $\sigma = 1$, and the results for this case are provided in Appendix~\ref{app:damped-taylor-tests}.
}

\begin{figure}[htpb]
\centering

  \includegraphics[width=.4\textwidth]{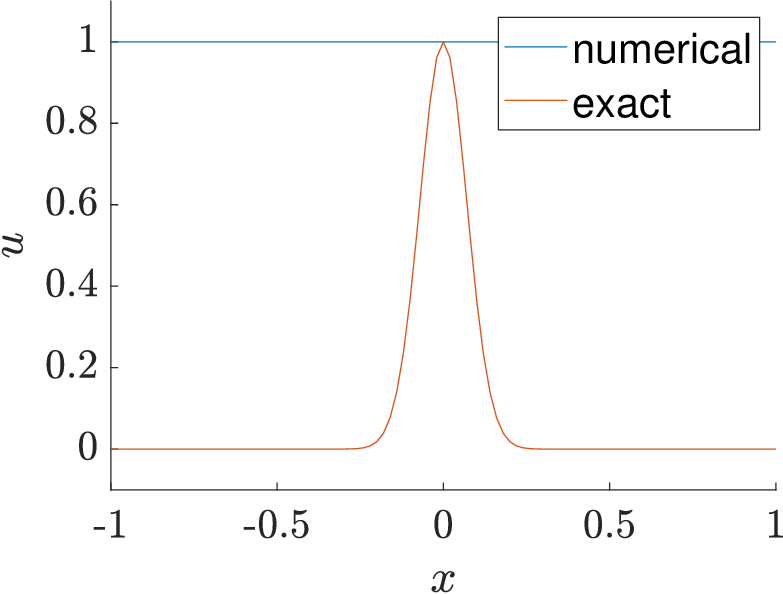}

\caption{The initial guess and the exact initial displacement profile to be recovered from boundary measurements in the 1D numerical inverse simulations.}
\label{fig:init_inverse_gaussian1D_neumann}
\end{figure}

\begin{figure}[htbp]
\centering
    \begin{subfigure}{.48\textwidth}
    \centering
    \includegraphics[width=\textwidth]{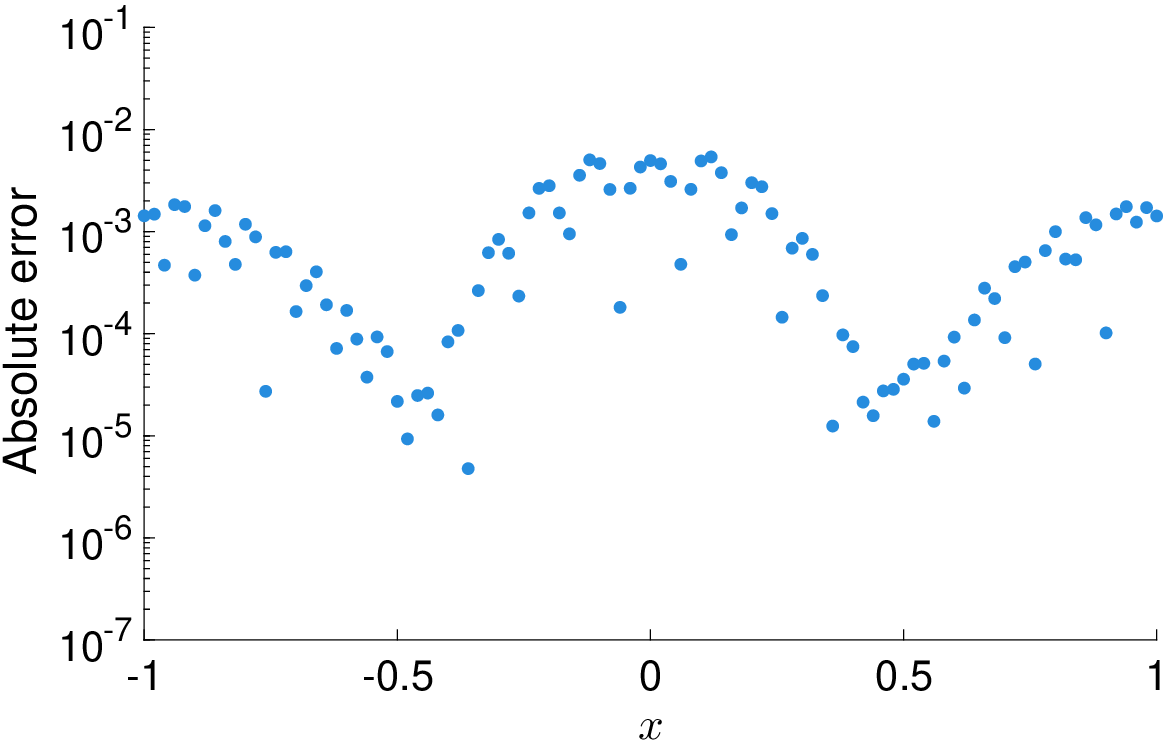}
    \caption{}
  \end{subfigure}
  \hfill
  \begin{subfigure}{.48\textwidth}
    \centering
    \includegraphics[width=\textwidth]{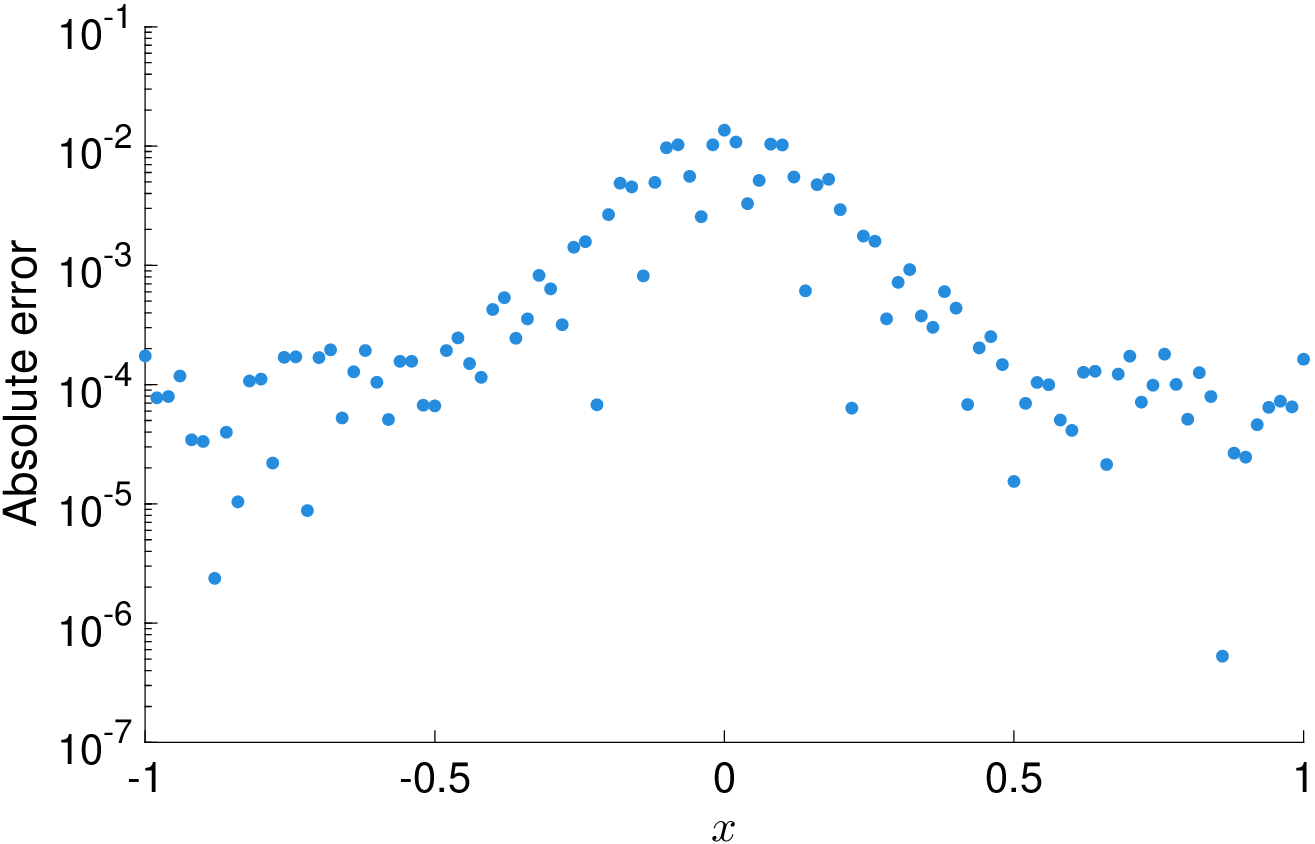}
    \caption{}
  \end{subfigure}

  \begin{subfigure}{.48\textwidth}
    \centering
    \resizebox{\textwidth}{!}{
    \begin{tikzpicture}
    \begin{loglogaxis}[
        width=\textwidth,
        xlabel={Iteration},
        ylabel={Objective functional},
        xtick={1,2,3,4,5},
        xticklabels={1,2,3,4,5},
        grid=major,
        legend pos=north east,
        legend cell align={left},
        mark size=2.5pt
    ]

    \addplot[
        color=blue,
        mark=*,
        thick
    ]
    coordinates {
        (1,4.540e-5)
        (2,5.399e-6)
        (3,1.259e-6)
        (4,5.805e-7)
        (5,2.920e-7)
    };
    \addlegendentry{DP-SBP}

    \addplot[
        color=red,
        mark=square*,
        thick
    ]
    coordinates {
        (1,1.303e-3)
        (2,6.809e-5)
        (3,3.088e-5)
        (4,2.934e-5)
        (5,2.916e-5)
    };
    \addlegendentry{RK4}

    \end{loglogaxis}
    \end{tikzpicture}}
    \caption{}
  \end{subfigure}
  \hfill
  \begin{subfigure}{.48\textwidth}
    \centering
    \resizebox{\textwidth}{!}{
    \begin{tikzpicture}
    \begin{loglogaxis}[
        width=\textwidth,
        xlabel={Iteration},
        ylabel={$L^2$ error},
        xtick={1,2,3,4,5},
        xticklabels={1,2,3,4,5},
        ymin=3.1623e-4,
        ymax=3.1623e-1,
        ytick={1e-3,1e-2,1e-1},
        grid=major,
        legend pos=north east,
        legend cell align={left},
        mark size=2.5pt
    ]

    \addplot[
        color=blue,
        mark=*,
        thick
    ]
    coordinates {
        (1,3.955e-3)
        (2,2.467e-3)
        (3,4.024e-3)
        (4,4.429e-3)
        (5,4.633e-3)
    };
    \addlegendentry{DP-SBP}

    \addplot[
        color=red,
        mark=square*,
        thick
    ]
    coordinates {
        (1,3.524e-2)
        (2,6.992e-3)
        (3,2.824e-3)
        (4,2.289e-3)
        (5,2.191e-3)
    };
    \addlegendentry{RK4}

    \end{loglogaxis}
    \end{tikzpicture}}
    \caption{}
  \end{subfigure}
\caption{{
  Iterative numerical inversion to recover the 1D Gaussian initial displacement field for the undamped wave equation $\sigma = 0$: (a) absolute error at iteration 2, (b) absolute error at iteration 5, (c) objective functional value vs iteration, (d) $L^2$ error vs iteration.}
}
\label{fig:inverse_gaussian1D_neumann}
\end{figure}

\begin{figure}[htbp]
\centering
    \begin{subfigure}{.48\textwidth}
    \centering
    \includegraphics[width=\textwidth]{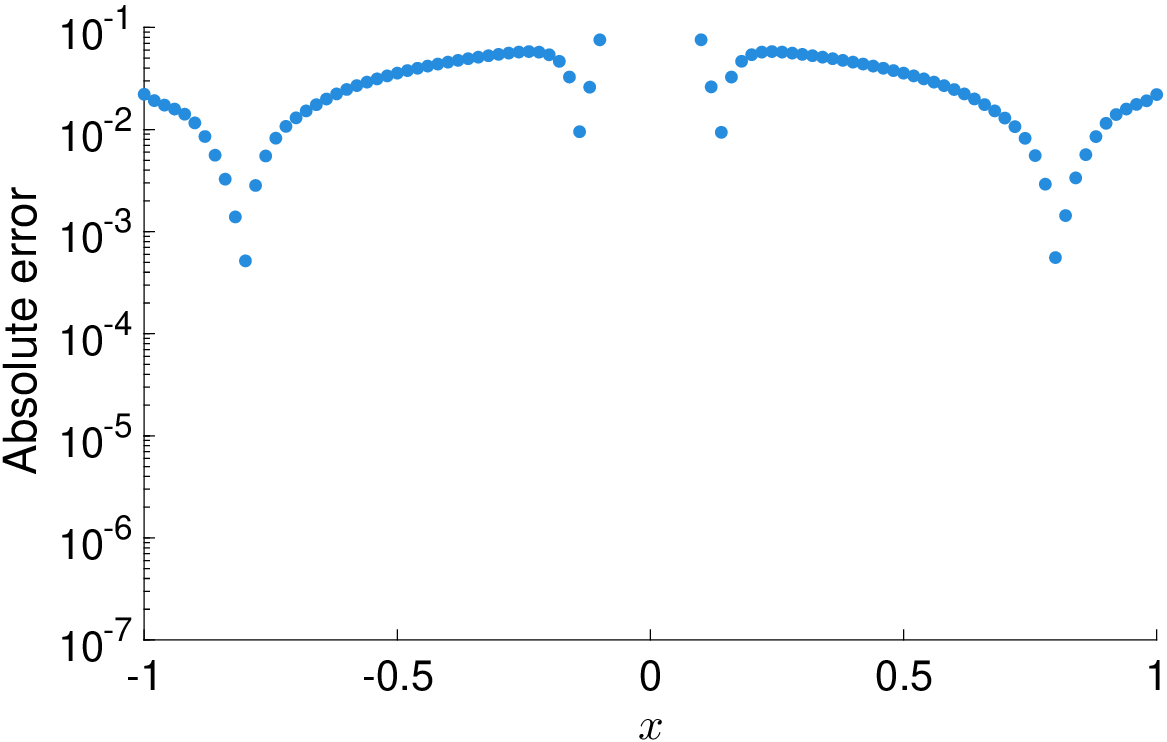}
    \caption{}
  \end{subfigure}
  \hfill
  \begin{subfigure}{.48\textwidth}
    \centering
    \includegraphics[width=\textwidth]{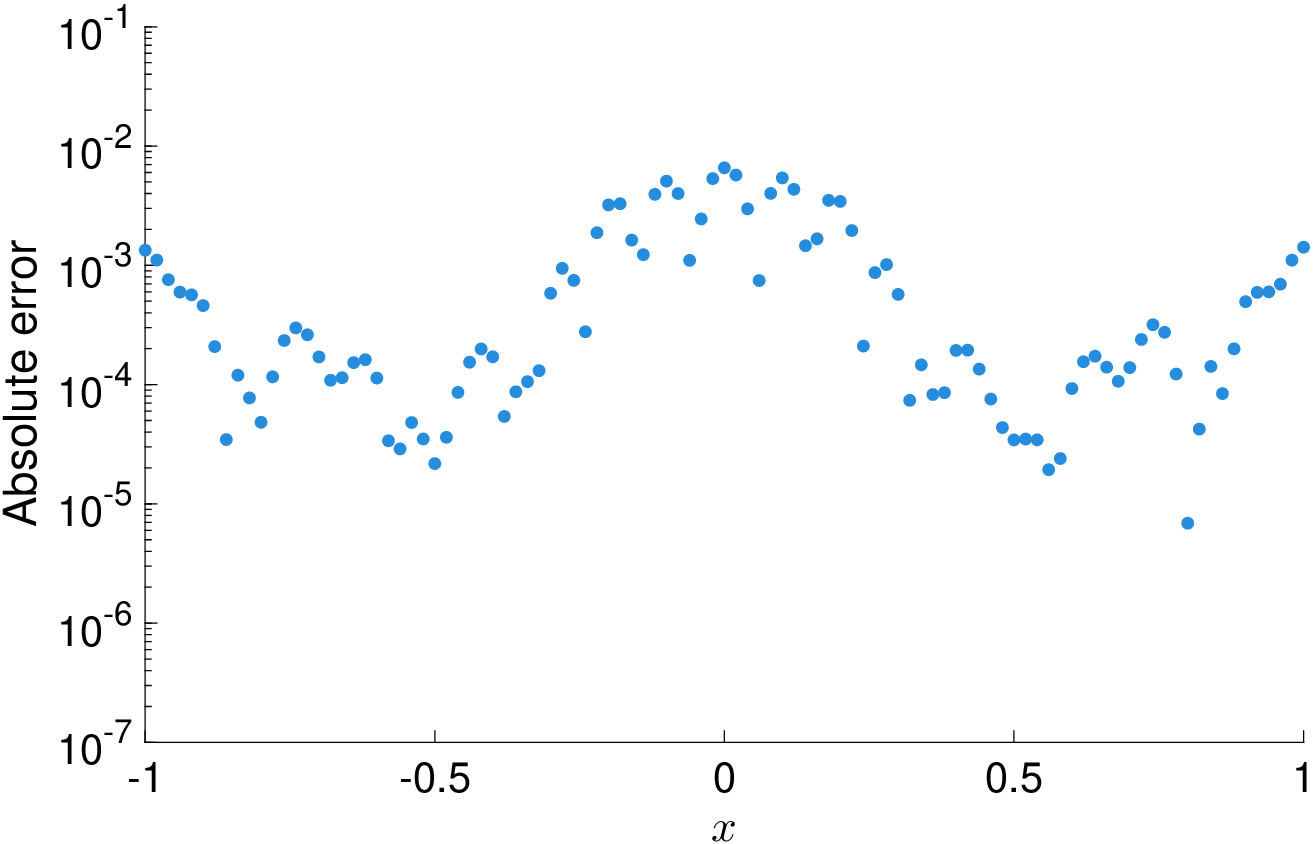}
    \caption{}
  \end{subfigure}

  \begin{subfigure}{.48\textwidth}
    \centering
    \resizebox{\textwidth}{!}{
    \begin{tikzpicture}
    \begin{loglogaxis}[
        width=\textwidth,
        xlabel={Iteration},
        ylabel={Objective functional},
        xtick={1,2,3,4,5},
        xticklabels={1,2,3,4,5},
        grid=major,
        legend pos=north east,
        legend cell align={left},
        mark size=2.5pt
    ]

    \addplot[
        color=blue,
        mark=*,
        thick
    ]
    coordinates {
        (1,1.375e-2)
        (2,5.288e-3)
        (3,1.672e-5)
        (4,6.201e-6)
        (5,9.230e-7)
    };
    \addlegendentry{DP-SBP}

    \addplot[
        color=red,
        mark=square*,
        thick
    ]
    coordinates {
        (1,1.388e-2)
        (2,5.337e-3)
        (3,3.302e-5)
        (4,1.685e-5)
        (5,1.133e-5)
    };
    \addlegendentry{RK4}

    \end{loglogaxis}
    \end{tikzpicture}}
    \caption{}
  \end{subfigure}
  \hfill
  \begin{subfigure}{.48\textwidth}
    \centering
    \resizebox{\textwidth}{!}{
    \begin{tikzpicture}
    \begin{loglogaxis}[
        width=\textwidth,
        xlabel={Iteration},
        ylabel={$L^2$ error},
        xtick={1,2,3,4,5},
        xticklabels={1,2,3,4,5},
        grid=major,
        legend pos=north east,
        legend cell align={left},
        mark size=2.5pt
    ]

    \addplot[
        color=blue,
        mark=*,
        thick
    ]
    coordinates {
        (1,1.901e-1)
        (2,1.178e-1)
        (3,4.898e-3)
        (4,2.657e-3)
        (5,2.514e-3)
    };
    \addlegendentry{DP-SBP}

    \addplot[
        color=red,
        mark=square*,
        thick
    ]
    coordinates {
        (1,1.909e-1)
        (2,1.186e-1)
        (3,7.224e-3)
        (4,4.496e-3)
        (5,2.701e-3)
    };
    \addlegendentry{RK4}

    \end{loglogaxis}
    \end{tikzpicture}}
    \caption{}
  \end{subfigure}
\caption{{
  Iterative numerical inversion to recover the 1D Gaussian initial displacement field for the damped wave equation $\sigma = 1$: (a) absolute error at iteration 2, (b) absolute error at iteration 5, (c) objective functional value vs iteration, (d) $L^2$ error vs iteration.}
}
\label{fig:inverse_gaussian1D_neumann2}
\end{figure}
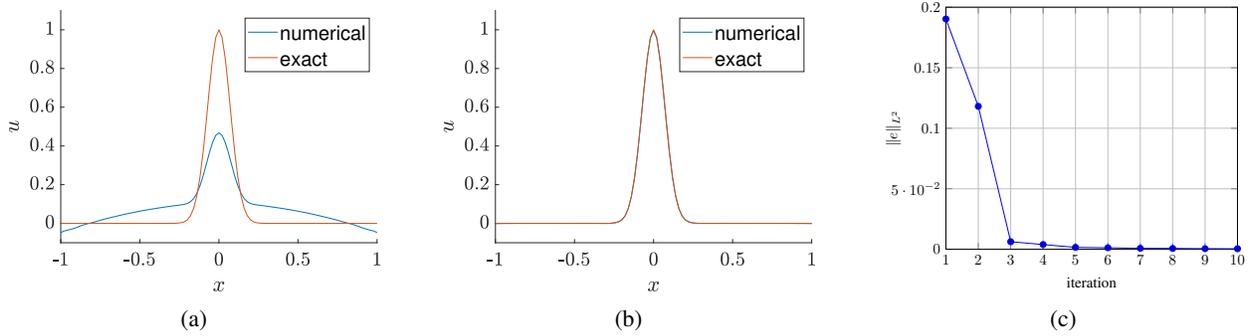

\begin{table}[htbp]
\centering
\begin{tabular}{c|cc|cc}
\hline
& \multicolumn{2}{c|}{DP-SBP}
& \multicolumn{2}{c}{RK4} \\
\cline{2-5}
$\varepsilon_i$
& $R(\varepsilon_i)$
& $\log_{10} \left(\tfrac{R(\varepsilon_{i-1})}{R(\varepsilon_i)} \right)$
& $R(\varepsilon_i)$
& $\log_{10} \left(\tfrac{R(\varepsilon_{i-1})}{R(\varepsilon_i)} \right)$ \\
\hline
$10^{-1}$
& $2.303133\times10^{-3}$
& ---
& $3.190551\times10^{-3}$
& --- \\

$10^{-2}$
& $2.303133\times10^{-5}$
& $2.0000$
& $6.579398\times10^{-5}$
& $1.6857$ \\

$10^{-3}$
& $2.303135\times10^{-7}$
& $2.0000$
& $4.046787\times10^{-6}$
& $1.2111$ \\

$10^{-4}$
& $2.303285\times10^{-9}$
& $2.0000$
& $3.793528\times10^{-7}$
& $1.0281$ \\

$10^{-5}$
& $2.315804\times10^{-11}$
& $1.9976$
& $3.768216\times10^{-8}$
& $1.0029$ \\

$10^{-6}$
& $3.936975\times10^{-13}$
& $1.7695$
& $3.765824\times10^{-9}$
& $1.0003$ \\
\hline
\end{tabular}
\caption{{Taylor remainders and observed convergence rates for the inverse problem of recovering the 1D Gaussian initial displacement field for the undamped wave equation $\sigma = 0$.}}
\label{tab:taylor_test_1D}
\end{table}

\subsection{Two-dimensional convergence test}

We consider the two-dimensional analogue of the IBVP \eqref{contwaveeq}, \eqref{initcond}, \eqref{wavebcs} in the spatial domain $\Omega = [-1,1]^2$ and the time interval $0 \leq t \leq T = 2$ with zero source term $S = 0$. We set the constant wave speed $c = 1$ and consider the 2D undamped wave equation $\sigma = 0$ and the 2D damped wave equation $\sigma = 1$. The initial conditions are $ u(x,y,0) = \cos(\pi x)\cos(\pi y)$ and $\left.\frac{\partial u}{\partial t}\right|_{t=0} = 0$, and we use homogeneous Neumann boundary conditions on $\partial \Omega$, analogous to those in subsection \ref{1d_convergence_test}. As in the one-dimensional case, the exact solutions are given in Appendix~\ref{appendix:exact_solutions}.

The numerical solution is computed using space-time DP-SBP operators of orders $p = 2,4,6,8$ with the grid parameters 
$\Delta x = \Delta y = \Delta t = 0.125, 0.0625, 0.03125$. 
For the finer time grids $\Delta t = 0.0625$ and $\Delta t = 0.03125$, the time interval $[0,2]$ is divided into two and four uniform blocks, and we apply the multiblock in time approach. 
We consider the discrete time derivative operators $D_{tt} \in {\{D^{-}_tD^{-}_t, D^{-}_tD^{\mathrm{o}}_t, D_t^\mathrm{o}D^{-}_t, D_t^\mathrm{o}D_t^\mathrm{o}\}}$, and compute the $L^2$ error  at the final time $T = 2$ for each case. 
The observed convergence rates are reported in Table \ref{tab:convergence_side_by_side_2D}. We omit the error plots since the trends are similar to the one-dimensional case.

\begin{table}[htpb]
\centering
\makebox[\textwidth]{%
  \begin{subtable}{0.48\textwidth}
    \centering
    \caption{}
    \resizebox{\textwidth}{!}{%
      \begin{tabular}{|c|cccc|}
        \hline
        \multirow{2}{*}{order} & \multicolumn{4}{c|}{$D_{tt}$} \\ \cline{2-5}
        & \multicolumn{1}{c|}{$D^-_tD^-_t$} & \multicolumn{1}{c|}{$D^-_t{D_t^\mathrm{o}}$} & \multicolumn{1}{c|}{${D_t^{\mathrm{o}}}D^-_t$} & ${D_t^{\mathrm{o}}D_t^{\mathrm{o}}}$ \\ \hline
        2 & \multicolumn{1}{c|}{2.7382} & \multicolumn{1}{c|}{2.6638} & \multicolumn{1}{c|}{2.5431} & 2.4524 \\ \hline
        4 & \multicolumn{1}{c|}{4.6016} & \multicolumn{1}{c|}{4.5690} & \multicolumn{1}{c|}{4.5858} & 4.5797 \\ \hline
        6 & \multicolumn{1}{c|}{4.8504} & \multicolumn{1}{c|}{4.8472} & \multicolumn{1}{c|}{4.8500} & 4.8648 \\ \hline
        8 & \multicolumn{1}{c|}{6.5534} & \multicolumn{1}{c|}{6.5540} & \multicolumn{1}{c|}{6.5658} & 6.5632 \\ \hline
      \end{tabular}
    }
  \end{subtable}
  \hspace{1em}
  \begin{subtable}{0.48\textwidth}
    \centering
    \caption{}
    \resizebox{\textwidth}{!}{%
      \begin{tabular}{|c|cccc|}
        \hline
        \multirow{2}{*}{order} & \multicolumn{4}{c|}{$D_{tt}$} \\ \cline{2-5}
        & \multicolumn{1}{c|}{$D^-_tD^-_t$} & \multicolumn{1}{c|}{$D^-_t{D_t^\mathrm{o}}$} & \multicolumn{1}{c|}{${D_t^{\mathrm{o}}}D^-_t$} & ${D_t^{\mathrm{o}}D_t^{\mathrm{o}}}$ \\ \hline
        2 & \multicolumn{1}{c|}{1.6901} & \multicolumn{1}{c|}{1.8148} & \multicolumn{1}{c|}{1.8403} & 1.9337 \\ \hline
        4 & \multicolumn{1}{c|}{3.7652} & \multicolumn{1}{c|}{3.8150} & \multicolumn{1}{c|}{3.7708} & 3.8334 \\ \hline
        6 & \multicolumn{1}{c|}{4.5145} & \multicolumn{1}{c|}{4.5036} & \multicolumn{1}{c|}{4.5036} & 4.4878 \\ \hline
        8 & \multicolumn{1}{c|}{6.8877} & \multicolumn{1}{c|}{6.8800} & \multicolumn{1}{c|}{6.8829} & 6.8594 \\ \hline
      \end{tabular}
    }
  \end{subtable}
}
\caption{Convergence rates for different second time derivative approximations and operator orders of accuracy for (a) the 2D undamped wave equation $\sigma = 0$ and for (b) the 2D damped wave equation $\sigma = 1$.}
\label{tab:convergence_side_by_side_2D}
\end{table}

\subsection{Two-dimensional Gaussian inverse problem} \label{numerical_inverse_2D} We consider the two-dimensional analogue of the one-dimensional inverse problem described in subsection \ref{numerical_inverse_1D}. Similar to that case, we choose $f(x,y) = e^{-8(x^2 + y^2)}$ as the initial displacement field, {and the boundary measurements are obtained from the exact solution of the two-dimensional analogue of the forward IBVP}.

We compute the numerical solution and the corresponding adjoint using the 4th order space–time DP-SBP operators with the grid parameters $\Delta x = 0.0625$ and $\Delta t = 0.03125$. The time interval $[0,2]$ is divided into four uniform blocks, and we apply the multiblock in time approach with the second time derivative operator $D_{tt} = D^-_tD^-_t$. {As in the one-dimensional case, we also consider the RK4 method for time integration with the same grid parameters. The gradient and optimization procedures are straightforward extensions of those described in subsection \ref{numerical_inverse_1D}, and we perform five BFGS iterations for both methods.} The initial guess of the BFGS iteration and the exact initial displacement field are shown in Fig. \ref{fig:init_inverse_gaussian2D_neumann}.

{
The absolute error snapshots, the objective functional values, and the $L^2$ errors are shown in Fig. \ref{fig:inverse_gaussian2D} for the undamped wave equation $\sigma=0$ and Fig. \ref{fig:inverse_gaussian2D_damped} for the damped wave equation $\sigma=1$. Similar to the one-dimensional case, the DP-SBP method achieves lower final objective functional values than RK4 for both cases. However, in the two-dimensional case, the DP-SBP method also yields lower final $L^2$ errors for both the undamped and damped wave equations, although the errors remain of the same order of magnitude. The absolute error snapshots show substantial reductions between iterations 2 and 5 for both the undamped and damped cases.
}

{Using the two-dimensional analogue of the Taylor remainder \eqref{taylor_remainder}, we also perform Taylor remainder tests following the procedure described in the one-dimensional case. As shown in Table \ref{tab:taylor_test_2D}, the Taylor remainder exhibits second order convergence for the DP-SBP method and first order convergence for the RK4 method for the undamped wave equation $\sigma = 0$. Similar convergence rates are observed for the damped wave equation $\sigma = 1$, and the results for this case are provided in Appendix~\ref{app:damped-taylor-tests}. These results are consistent with those observed in the one-dimensional case.
}

\begin{figure}[htb]
\centering
    \begin{subfigure}{.4\textwidth}
    \centering
    \includegraphics[width=\textwidth]{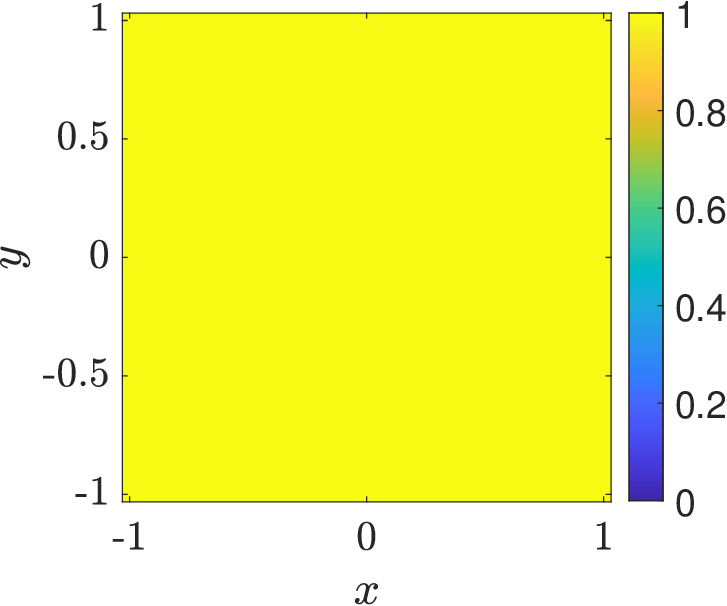}
    \caption{}
  \end{subfigure}
  \qquad
  \begin{subfigure}{.4\textwidth}
    \centering
    \includegraphics[width=\textwidth]{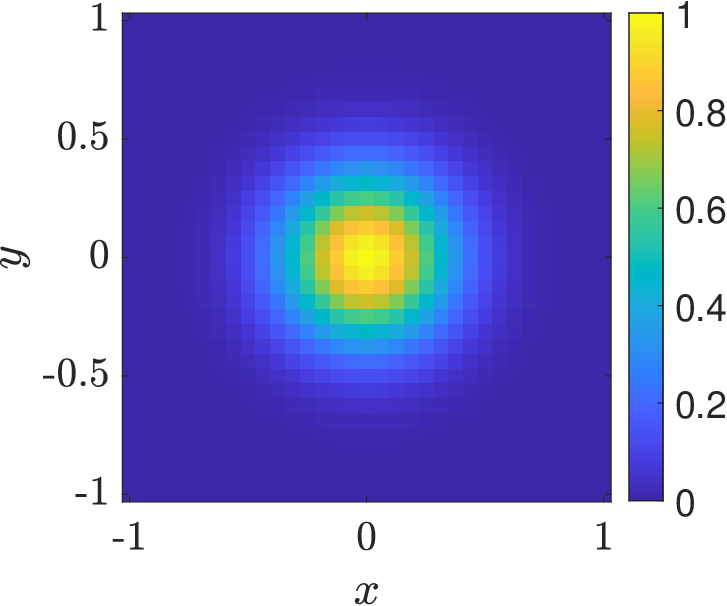}
    \caption{}
  \end{subfigure}
\caption{(a) The initial guess and (b) the exact initial displacement profile to be recovered from boundary measurements in the 2D numerical inverse simulations.}
\label{fig:init_inverse_gaussian2D_neumann}
\end{figure}

\begin{figure}[htbp]
\centering
  \begin{subfigure}{.48\textwidth}
    \centering
    \includegraphics[width=\textwidth]{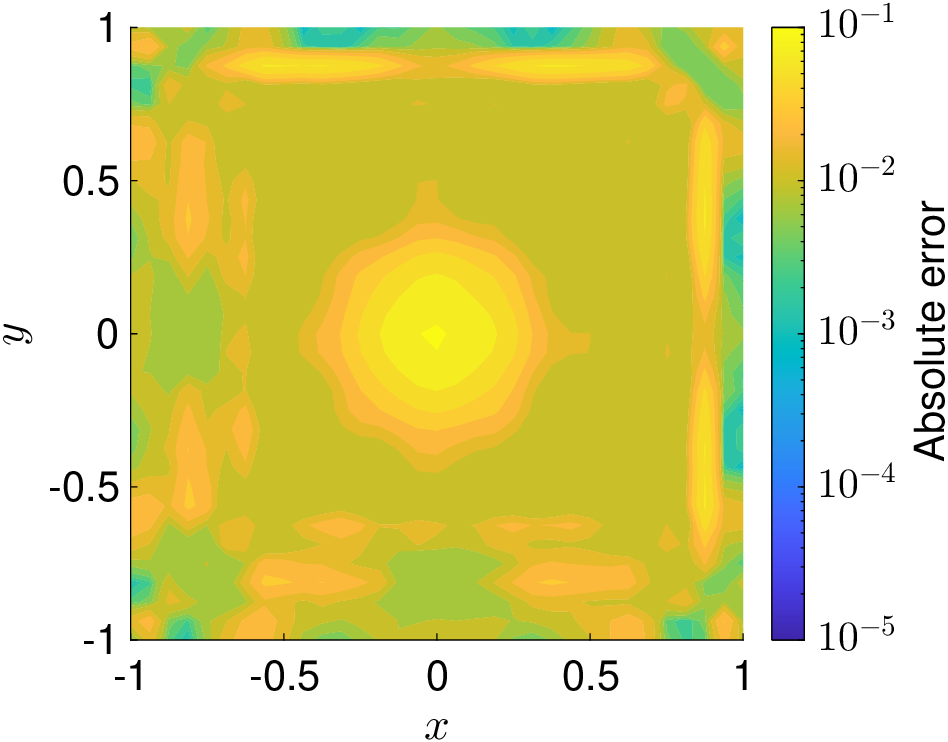}
    \caption{}
  \end{subfigure}
  \hfill
  \begin{subfigure}{.48\textwidth}
    \centering
    \includegraphics[width=\textwidth]{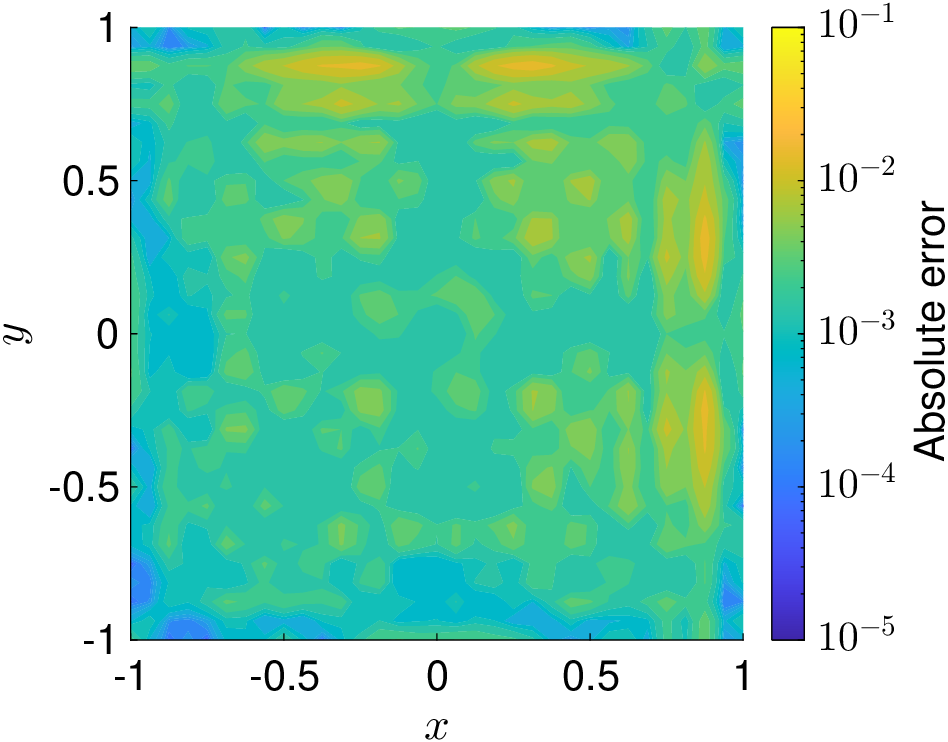}
    \caption{}
  \end{subfigure}

  \begin{subfigure}{.48\textwidth}
    \centering
    \resizebox{\textwidth}{!}{
    \begin{tikzpicture}
    \begin{loglogaxis}[
        width=\textwidth,
        xlabel={Iteration},
        ylabel={Objective functional},
        xtick={1,2,3,4,5},
        xticklabels={1,2,3,4,5},
        grid=major,
        legend pos=north east,
        legend cell align={left},
        mark size=2.5pt
    ]

    \addplot[
        color=blue,
        mark=*,
        thick
    ]
    coordinates {
        (1,2.715e-2)
        (2,3.075e-3)
        (3,1.981e-4)
        (4,8.392e-5)
        (5,2.762e-5)
    };
    \addlegendentry{DP-SBP}

    \addplot[
        color=red,
        mark=square*,
        thick
    ]
    coordinates {
        (1,3.678e-2)
        (2,1.349e-2)
        (3,4.837e-3)
        (4,1.797e-3)
        (5,6.115e-4)
    };
    \addlegendentry{RK4}

    \end{loglogaxis}
    \end{tikzpicture}}
    \caption{}
  \end{subfigure}
  \hfill
  \begin{subfigure}{.48\textwidth}
    \centering
    \resizebox{\textwidth}{!}{
    \begin{tikzpicture}
    \begin{loglogaxis}[
        width=\textwidth,
        xlabel={Iteration},
        ylabel={$L^2$ error},
        xtick={1,2,3,4,5},
        xticklabels={1,2,3,4,5},
        grid=major,
        legend pos=north east,
        legend cell align={left},
        mark size=2.5pt
    ]

    \addplot[
        color=blue,
        mark=*,
        thick
    ]
    coordinates {
        (1,1.342e-1)
        (2,4.797e-2)
        (3,1.655e-2)
        (4,1.129e-2)
        (5,7.445e-3)
    };
    \addlegendentry{DP-SBP}

    \addplot[
        color=red,
        mark=square*,
        thick
    ]
    coordinates {
        (1,1.456e-1)
        (2,6.739e-2)
        (3,3.696e-2)
        (4,2.293e-2)
        (5,1.489e-2)
    };
    \addlegendentry{RK4}

    \end{loglogaxis}
    \end{tikzpicture}}
    \caption{}
  \end{subfigure}

\caption{{
  Iterative numerical inversion to recover the 2D Gaussian initial
  displacement field for the undamped wave equation $\sigma=0$:
  (a) absolute error at iteration 2, (b) absolute error at iteration 5, (c) objective functional value
  vs iteration, (d) $L^2$ error vs iteration.}
}
\label{fig:inverse_gaussian2D}
\end{figure}


\begin{figure}[htbp]
\centering
  \begin{subfigure}{.48\textwidth}
    \centering
    \includegraphics[width=\textwidth]{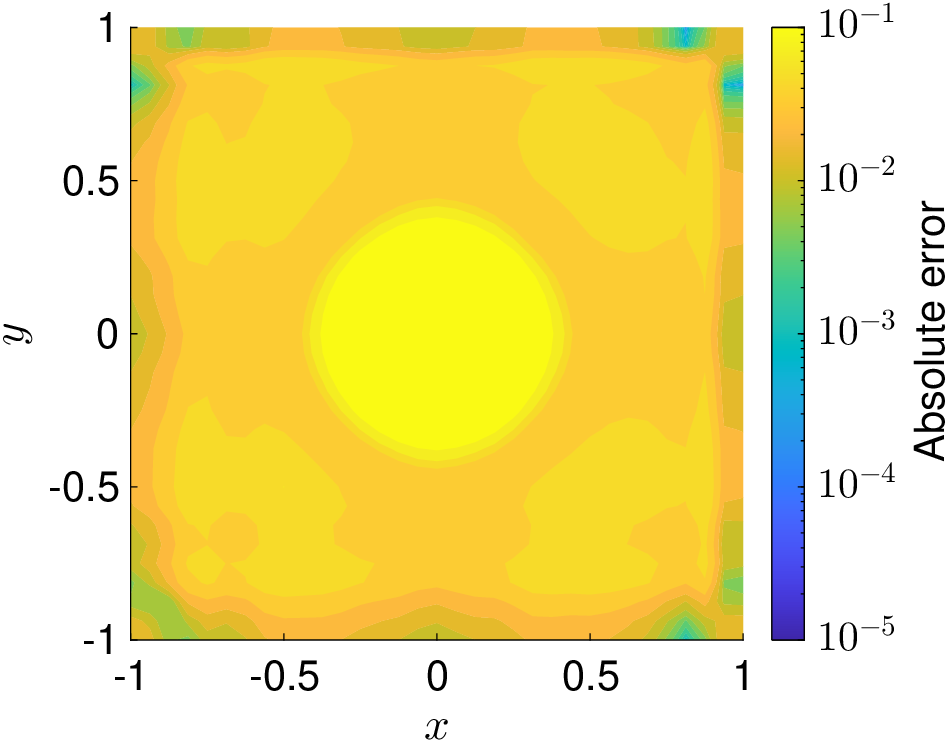}
    \caption{}
  \end{subfigure}
  \hfill
  \begin{subfigure}{.48\textwidth}
    \centering
    \includegraphics[width=\textwidth]{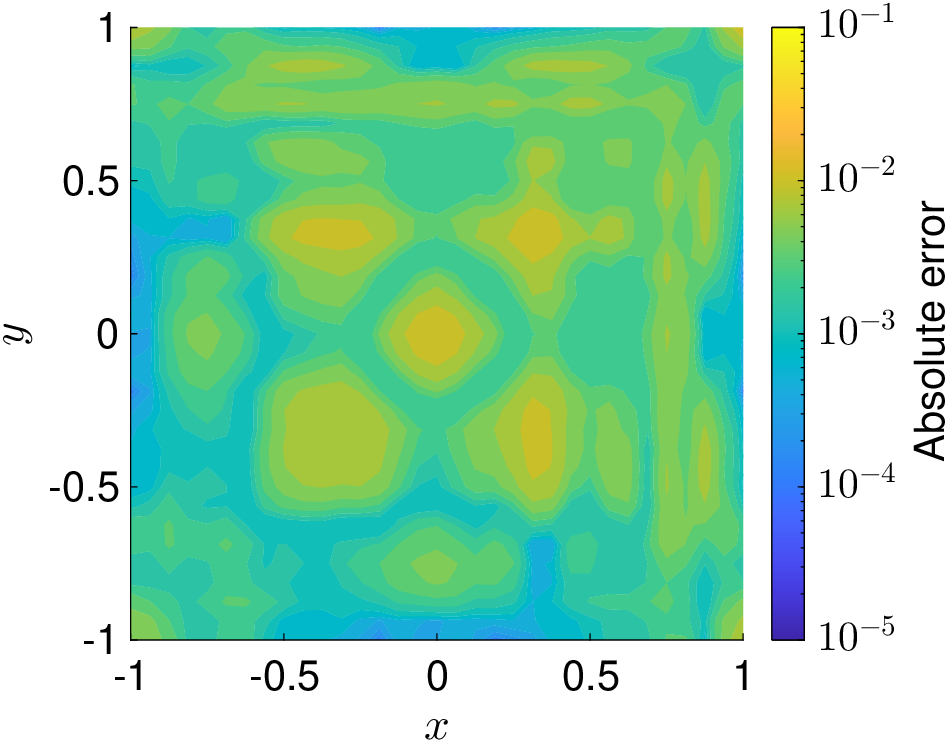}
    \caption{}
  \end{subfigure}

  \begin{subfigure}{.48\textwidth}
    \centering
    \resizebox{\textwidth}{!}{
    \begin{tikzpicture}
    \begin{loglogaxis}[
        width=\textwidth,
        xlabel={Iteration},
        ylabel={Objective functional},
        xtick={1,2,3,4,5},
        xticklabels={1,2,3,4,5},
        grid=major,
        legend pos=north east,
        legend cell align={left},
        mark size=2.5pt
    ]

    \addplot[
        color=blue,
        mark=*,
        thick
    ]
    coordinates {
        (1,3.618e-2)
        (2,1.998e-2)
        (3,5.360e-4)
        (4,1.988e-4)
        (5,3.077e-5)
    };
    \addlegendentry{DP-SBP}

    \addplot[
        color=red,
        mark=square*,
        thick
    ]
    coordinates {
        (1,3.660e-2)
        (2,2.015e-2)
        (3,1.351e-3)
        (4,5.266e-4)
        (5,1.312e-4)
    };
    \addlegendentry{RK4}

    \end{loglogaxis}
    \end{tikzpicture}}
    \caption{}
  \end{subfigure}
  \hfill
  \begin{subfigure}{.48\textwidth}
    \centering
    \resizebox{\textwidth}{!}{
    \begin{tikzpicture}
    \begin{loglogaxis}[
        width=\textwidth,
        xlabel={Iteration},
        ylabel={$L^2$ error},
        xtick={1,2,3,4,5},
        xticklabels={1,2,3,4,5},
        grid=major,
        legend pos=north east,
        legend cell align={left},
        mark size=2.5pt
    ]

    \addplot[
        color=blue,
        mark=*,
        thick
    ]
    coordinates {
        (1,2.684e-1)
        (2,2.013e-1)
        (3,2.929e-2)
        (4,1.915e-2)
        (5,8.942e-3)
    };
    \addlegendentry{DP-SBP}

    \addplot[
        color=red,
        mark=square*,
        thick
    ]
    coordinates {
        (1,2.692e-1)
        (2,2.019e-1)
        (3,3.395e-2)
        (4,2.200e-2)
        (5,1.591e-2)
    };
    \addlegendentry{RK4}

    \end{loglogaxis}
    \end{tikzpicture}}
    \caption{}
  \end{subfigure}

\caption{{
  Iterative numerical inversion to recover the 2D Gaussian initial
  displacement field for the damped wave equation $\sigma=1$:
  (a) absolute error at iteration 2, (b) absolute error at iteration 5, (c) objective functional value
  vs iteration, (d) $L^2$ error vs iteration.}
}
\label{fig:inverse_gaussian2D_damped}
\end{figure}

\begin{table}[htbp]
\centering
\begin{tabular}{c|cc|cc}
\hline
& \multicolumn{2}{c|}{DP-SBP}
& \multicolumn{2}{c}{RK4} \\
\cline{2-5}
$\varepsilon_i$
& $R(\varepsilon_i)$
& $\log_{10} \left(\tfrac{R(\varepsilon_{i-1})}{R(\varepsilon_i)} \right)$
& $R(\varepsilon_i)$
& $\log_{10} \left(\tfrac{R(\varepsilon_{i-1})}{R(\varepsilon_i)} \right)$ \\
\hline
$10^{-1}$
& $5.996313\times10^{-3}$
& ---
& $1.154319\times10^{-2}$
& --- \\

$10^{-2}$
& $5.996313\times10^{-5}$
& $2.0000$
& $1.163389\times10^{-4}$
& $1.9966$ \\

$10^{-3}$
& $5.996316\times10^{-7}$
& $2.0000$
& $1.254087\times10^{-6}$
& $1.9674$ \\

$10^{-4}$
& $5.996615\times10^{-9}$
& $2.0000$
& $2.161081\times10^{-8}$
& $1.7637$ \\

$10^{-5}$
& $6.029493\times10^{-11}$
& $1.9976$
& $1.123202\times10^{-9}$
& $1.2842$ \\

$10^{-6}$
& $9.369446\times10^{-13}$
& $1.8086$
& $1.020375\times10^{-10}$
& $1.0417$ \\
\hline
\end{tabular}
\caption{{Taylor remainders and observed convergence rates for the inverse problem of recovering the 2D Gaussian initial displacement field for the undamped wave equation $\sigma = 0$.}}
\label{tab:taylor_test_2D}
\end{table}

\section{Conclusion}

In this paper, we have extended the DP-SBP framework in time, using the wave equation with damping as the model problem. We have considered both the forward IBVP and the corresponding adjoint IBVP. Our key contributions are systematically integrating the DP-SBP operators in space and time, proving fully discrete stability, and ensuring fully discrete adjoint consistency, all of which represent novel advances.

We have derived semi-discrete and fully discrete numerical approximations of the IBVP using the DP-SBP operators. For each case, energy estimates mimicking the continuous energy estimates have been derived. We have also obtained the corresponding semi-discrete and fully discrete adjoint approximations and shown adjoint consistency and energy stability.

Numerical experiments in both one and two spatial dimensions have been presented to demonstrate both the efficacy and accuracy of the proposed numerical scheme. Furthermore, numerical simulations of initial displacement field recovery from boundary measurements have been presented to demonstrate the potential of the numerical scheme for solving inverse problems.

\appendix

\section{Proof of Lemma \ref{lem:lemmata}}
\label{app:proof-stability-lemma}

We multiply \eqref{auxdiscrete} by $\mathbf{U}^\top(H_t \otimes (D_x^{-})^\top H_xD_x^{-})$ and \eqref{auxdiscrete2} by $\mathbf{V}^\top(H_t \otimes (\mathbf{c}^{-1})^2H_x)$ and add them together. Then, we complete the squares by adding and subtracting the terms $\frac{1}{2}\frac{\mu_1^2}{2\mu_1 + 1} \left\|D^-_x\mathbf{f}\right\|_{H_x}^2$ and $\frac{1}{2}\frac{\mu_3^2}{2\mu_3 + 1}\left\|\mathbf{c}^{-1}\mathbf{g}\right\|_{H_x}^2$. We have
{\small
\begin{align*}
&\mathcal{E}_M +\sum_{k=1}^M h_k \left(\|\mathbf{c}^{-1}\pmb{\sigma}\mathbf{V}_k \|_{H_x}^2+ \operatorname{BT}_{hk} \right) \Delta t + \frac{1}{2}\mathbf{U}^{\top}\left( S^{j}_t \otimes (D^-_x)^{\top} H_x (D^-_x) \right)\mathbf{U} \\
&+ \frac{1}{2}\mathbf{V}^{\top}\left(S^{i}_t \otimes  (\mathbf{c}^{-1})^2H_x \right)   \mathbf{V}  \\
=&\sum_{k=1}^M h_k\mathbf{V}_k^{\top} ((\mathbf{c}^{-1})^2H_x)\mathbf{S}_k\Delta{t} + \frac{1}{2}\left(2\mu_1 + 1\right) \left\|D^-_x\left(\mathbf{U}_1-\frac{\mu_1}{2\mu_1 + 1}\mathbf{f}\right)\right\|_{H_x}^2 
\\
&-\frac{1}{2}\frac{\mu_1^2}{2\mu_1 + 1} \left\|D^-_x\mathbf{f}\right\|_{H_x}^2 + \frac{1}{2}\left(2\mu_3 + 1\right)\left\|\mathbf{c}^{-1}\left(\mathbf{V}_1-\frac{\mu_3}{2\mu_3 + 1}\mathbf{g}\right)\right\|_{H_x}^2 
- \frac{1}{2}\frac{\mu_3^2}{2\mu_3 + 1}\left\|\mathbf{c}^{-1}\mathbf{g}\right\|_{H_x}^2.
\end{align*}
}%
For convenience, we set
{\small
\begin{align*}
\mathcal{N} &=\frac{1}{2}\mathbf{U}^{\top}\left( S^{j}_t \otimes (D^-_x)^{\top} H_x (D^-_x) \right)\mathbf{U}  + \frac{1}{2}\mathbf{V}^{\top}\left(S^{i}_t \otimes  (\mathbf{c}^{-1})^2H_x \right)   \mathbf{V} \\
&- \frac{1}{2}\left(2\mu_1 + 1\right) \left\|D^-_x\left(\mathbf{U}_1-\frac{\mu_1}{2\mu_1 + 1}\mathbf{f}\right)\right\|_{H_x}^2- \frac{1}{2}\left(2\mu_3 + 1\right)\left\|\mathbf{c}^{-1}\left(\mathbf{V}_1-\frac{\mu_3}{2\mu_3 + 1}\mathbf{g}\right)\right\|_{H_x}^2,
\end{align*}
}%
which gives
{\small
\begin{align*}
&\mathcal{E}_M +  \sum_{k=1}^M h_k \left(\|\mathbf{c}^{-1}\pmb{\sigma}\mathbf{V}_k \|_{H_x}^2+ \operatorname{BT}_{hk} \right) \Delta t + \mathcal{N} \\
&\quad=- \frac{1}{2}\frac{\mu_3^2}{2\mu_3 + 1}\left\|\mathbf{c}^{-1}\mathbf{g}\right\|_{H_x}^2
- \frac{1}{2}\frac{\mu_1^2}{2\mu_1 + 1} \left\|D^-_x\mathbf{f}\right\|_{H_x}^2 +\sum_{k=1}^M h_k\mathbf{V}_k^{\top} ((\mathbf{c}^{-1})^2H_x)\mathbf{S}_k\Delta{t}.  
\end{align*}
}%
Since $2\mu_1 + 1 < 0$ and $2\mu_3 + 1 < 0$, Assumption B.1 implies $\mathcal{N} \geq 0$. Hence, using Cauchy-Schwarz inequality to the source term yields 
{\small
\begin{align*}
&\mathcal{E}_M +  \sum_{k=1}^M h_k \left(\|\mathbf{c}^{-1}\pmb{\sigma}\mathbf{V}_k \|_{H_x}^2+ \operatorname{BT}_{hk} \right) \Delta t 
\leq C \mathcal{E}_0 + \sum_{k=1}^M h_k \|\mathbf{c}^{-1}\mathbf{V}_k \|_{H_x} \|\mathbf{c}^{-1}\mathbf{S}_k\|_{H_x}\Delta t, 
\end{align*}
}where $C = \max\left(\frac{-\mu_1^2}{2\mu_1 + 1}, \frac{-\mu_3^2}{2\mu_3 + 1}\right)> 0$. The proof is complete.

\section{Error estimate}
\label{app:error-estimate}

Before deriving the error estimate, we investigate the accuracy of the second derivative DP-SBP operators by extending previous works \cite{mattsson2017diagonal, williams2024full}. Table \ref{second_derivative_accuracy} presents the interior accuracy of the second derivative DP-SBP operators and the boundary accuracy of the corresponding first derivative operators.

\begin{table}[H]
\centering
\begin{tabular}{lcccc}
\hline
Second derivative interior & $2$ & $4$ & $6$ & $8$ \\
\hline
First derivative boundary & $1$ & $2$ & $3$ & $4$ \\
\hline
\end{tabular}
\caption{Accuracy orders of DP-SBP operators $D^iD^j$ for $i,j\in\{-,\mathrm{o},+\}$.}
\label{second_derivative_accuracy}
\end{table}

Let us now derive the error estimate. We define the fully discrete error by
\[
\mathbf{E} = \mathbf{U} - \mathbf{U}_{\mathrm{exact}},
\]
where $\mathbf{U}_{\mathrm{exact}}$ denotes the exact solution evaluated at the space-time grid points. Using the stable penalty parameters $\mu_1 = \mu_3 = -1$, we also define the fully discrete energy norm error
\begin{equation} \label{fullydiscreteenergyerror}
\widehat{\mathcal{E}}_{k} = \left \{
\begin{array}{rl}
 \frac{1}{2}\left( \left\| \mathbf{c}^{-1}\left[(\widetilde{D}_t^{j}\otimes I_x)\mathbf{E}\right]_k \right\|_{H_x}^2 + \left\|D^-_x\mathbf{E}_k \right\|^2_{H_x}\right) ,  & \text{if} \quad k = 1,2,\ldots,M,\\
0,
 & \text{if} \quad k = 0.
\end{array} \right.
\end{equation}%
The fully discrete numerical approximation \eqref{fullydiscrete2} satisfies the error equation
\begin{align}
\begin{split}
&(\widetilde{D}^{i}_t\widetilde{D}^{j}_t \otimes I_x)\mathbf{E}
+(\widetilde{D}^{j}_t \otimes (\pmb{\sigma}^2 + B))\mathbf{E}
-(I_t \otimes \mathbf{c}^2\widetilde{D}_{xx})\mathbf{E}
= -\mathbf{T},
\end{split}
\end{align}
where $\mathbf{T}$ is the total truncation error given by
{\small
$$
\begin{aligned}
&\mathbf T={}\left[(D_t^{i}D_t^{j}\otimes I_x)\mathbf U_{\mathrm{exact}}-\frac{\partial^2\mathbf U_{\mathrm{exact}}}{\partial t^2}\right]\\
&+(I_t\otimes\pmb{\sigma}^2)\left[(D_t^{j}\otimes I_x)\mathbf U_{\mathrm{exact}}-\frac{\partial\mathbf U_{\mathrm{exact}}}{\partial t}\right]\\
&-(I_t\otimes\mathbf c^2)\left[(I_t\otimes D_x^+D_x^-)\mathbf U_{\mathrm{exact}}-\frac{\partial^2\mathbf U_{\mathrm{exact}}}{\partial x^2}\right]\\
&+(I_t\otimes H_x^{-1}\mathbf c^2e_Ne_N^\top)\left[(I_t\otimes D_x^-)\mathbf U_{\mathrm{exact}}-\frac{\partial\mathbf U_{\mathrm{exact}}}{\partial x}+\frac{\beta_R}{\alpha_R}\left[(D_t^{j}\otimes I_x)\mathbf U_{\mathrm{exact}}-\frac{\partial\mathbf U_{\mathrm{exact}}}{\partial t}\right]\right]\\
&-(I_t\otimes H_x^{-1}\mathbf c^2e_1e_1^\top)\left[(I_t\otimes D_x^-)\mathbf U_{\mathrm{exact}}-\frac{\partial\mathbf U_{\mathrm{exact}}}{\partial x}+\frac{\beta_L}{\alpha_L}\left[(D_t^{j}\otimes I_x)\mathbf U_{\mathrm{exact}}-\frac{\partial\mathbf U_{\mathrm{exact}}}{\partial t}\right]\right] \\
&+(H_t^{-1}e_1e_1^\top \otimes I_x)\left[(D_t^{j}\otimes I_x)\mathbf U_{\mathrm{exact}}-\mathbf{G}\right].
\end{aligned}
$$
}%
Applying the energy estimate theorem gives
$$
\left(\widehat{\mathcal{E}}_M\right)^{1/2}
\leq
C_1\| \mathbf{T} \|_{H_t \otimes H_x},
$$
where $C_1 > 0$ is a non-dimensional constant independent of the space-time grid parameters. Hence, assuming that $\Delta t = C_2 \Delta x$ for some constant $C_2 > 0$, we use the triangle inequality to obtain
$$
\left(\widehat{\mathcal{E}}_M\right)^{1/2} = O(\Delta x^{\lfloor p/2 \rfloor - 1/2}).
$$
Here, we have used the fact that the first derivative DP-SBP operator has boundary accuracy of order $\lfloor p/2\rfloor$ for the second derivative DP-SBP operators with interior accuracy of order $p$, according to Table \ref{second_derivative_accuracy}.

\section{Time multiblock}
\label{appendix:multiblock}

For simplicity, we divide the time domain $[0,T]$ into two uniform blocks, $[0,T/2]$ and $[T/2,T]$, each containing the same number of degrees of freedom $M$. We denote the degrees of freedom in the two uniform blocks by $\mathbf{U}^{(1)}$ and $\mathbf{U}^{(2)}$, respectively. Utilizing \eqref{fullydiscrete2}, when the stable penalty parameters $\mu_1 = \mu_3 = -1$ are used, we have the following fully discrete multiblock in time approximation:
\begin{align}
\begin{split} \label{fullydiscretemultiblock}
    &(\widetilde{D}^{i}_t\widetilde{D}^{j}_t \otimes I_x)\mathbf{U}^{(1)} + (\widetilde{D}^{j}_t \otimes (\pmb{\sigma}^2 + {B}))\mathbf{U}^{(1)} - (I_t \otimes \mathbf{c}^2\widetilde{D}_{xx})\mathbf{U}^{(1)} = \widetilde{\operatorname{SAT}}^{(1)} + \bar{\mathbf{S}}^{(1)} ,\\
    &(\widetilde{D}^{i}_t\widetilde{D}^{j}_t \otimes I_x)\mathbf{U}^{(2)} + (\widetilde{D}^{j}_t \otimes (\pmb{\sigma}^2 + {B}))\mathbf{U}^{(2)} - (I_t \otimes \mathbf{c}^2\widetilde{D}_{xx})\mathbf{U}^{(2)} = \widetilde{\operatorname{SAT}}^{(2)} + \bar{\mathbf{S}}^{(2)},
\end{split}    
\end{align}
where
\begin{align}
\begin{split} \label{fullydiscretemultiSAT}
\widetilde{\operatorname{SAT}}^{(1)} &= 
\left(\widetilde{D}_t^{i} H_t^{-1} e_1 e_1^\top \otimes I_x \right)\mathbf{F}
+ \left(H_t^{-1} e_1 e_1^\top \otimes (\pmb{\sigma}^2 + B)\right)\mathbf{F}
+ \left(H_t^{-1} e_1 e_1^\top \otimes I_x\right)\mathbf{G}, \\
\widetilde{\operatorname{SAT}}^{(2)} &=
\left(\widetilde{D}_t^{i} H_t^{-1} e_1 e_M^\top \otimes I_x \right)\mathbf{U}^{(1)}
+ \left(H_t^{-1} e_1 e_M^\top \otimes (\pmb{\sigma}^2 + B)\right)\mathbf{U}^{(1)} \\
&\quad
+ \left(H_t^{-1} e_1 e_M^\top \widetilde{D}_t^{j} \otimes I_x\right)\mathbf{U}^{(1)}.
\end{split}
\end{align}
Here, $\bar{\mathbf{S}}^{(1)}$ and $\bar{\mathbf{S}}^{(2)}$ are obtained by evaluating $\widetilde{\mathbf{S}}(t)$ at the time grid points in $[0,T/2]$ and $[T/2,T]$, respectively.

\section{Exact solutions used in the convergence tests}
\label{appendix:exact_solutions}

We assume that $c > 0$ and $\sigma \geq 0$ are constants. Let
\begin{equation*}
\phi(t, \omega)=
\begin{cases}
\cos\left(\frac{t}{2}\sqrt{\omega}\right) + \frac{\sigma^2}{\sqrt{\omega}}\sin\left(\frac{t}{2}\sqrt{\omega}\right),
& \text{if} \quad \omega > 0, \\
1+\frac{\sigma^2}{2}t,
& \text{if} \quad \omega = 0, \\
\cosh\left(\frac{t}{2}\sqrt{-\omega}\right) + \frac{\sigma^2}{\sqrt{-\omega}}\sinh\left(\frac{t}{2}\sqrt{-\omega}\right),
& \text{if} \quad \omega < 0.
\end{cases}
\end{equation*}
The exact solutions of the test problems in the one- and two-dimensional convergence tests can be written as
\begin{equation*}
u(x,t) = e^{-\frac{\sigma^2}{2}t}\phi(t, \omega_{1D})\cos(\pi x) \quad \text{and} \quad u(x,y,t) = e^{-\frac{\sigma^2}{2}t}\phi(t, \omega_{2D})\cos(\pi x) \cos(\pi y),
\end{equation*}
respectively, where 
\begin{equation*}
\omega_{1D} = 4c^2 \pi^2 - \sigma^4 \quad \text{and} \quad \omega_{2D} = 8c^2\pi^2 - \sigma^4.
\end{equation*}

\section{Gradient and optimization with MATLAB's \texttt{fminunc}}
\label{appendix:gradient}

Before describing how MATLAB's \texttt{fminunc} is used in the optimization, we first compute the gradient with respect to the initial displacement vector $\mathbf{f}$. For simplicity, we only consider the one-dimensional case with the observation weights $\rho = 0$ and $\theta = 1$. Recall that in this case, the fully discrete objective functional is
\begin{equation} \label{fullydiscretegaussianobj}
\mathbf{J}(\mathbf{U}) = \frac{1}{2}\left(\mathbf{U} - \mathbf{U}^{\mathrm{obs}} \right)^\top\left(H_t \otimes (e_Ne_N^\top + e_1e_1^\top) \right)\left(\mathbf{U} - \mathbf{U}^{\mathrm{obs}} \right).
\end{equation}%
Following the adjoint state method, we compute the gradient with respect to the initial displacement vector $\mathbf{f}$ by utilizing the Lagrangian \eqref{fullydiscretelagrangian} as follows:
\begin{equation} \label{fullydiscretegradient}
    \nabla_{\mathbf{f}} \mathbf{J}(\mathbf{U}) = \nabla_{\mathbf{f}} \pmb{\mathbf{L}}(\mathbf{U}, \pmb{\Lambda}) = \left(\left(e_1^\top \widehat{D}_t^{i}\otimes (\mathbf{c}^{-1})^2\right) - \left(e_1^\top \otimes (\mathbf{c}^{-1})^2(\pmb{\sigma}^2 + B)\right)\right)\pmb{\Lambda}. 
\end{equation}

Since MATLAB's \texttt{fminunc} performs its optimization in the Euclidean inner product, we introduce the reparameterization $\pmb{\eta} = H_x^{1/2} \mathbf{f}$ to ensure inner product consistency. The pseudocode in Algorithm \ref{alg:objgrad} provides an implementation of the function that returns both the objective \eqref{fullydiscretegaussianobj} and the scaled gradient $H_x^{1/2} \nabla_{\mathbf{f}} \mathbf{J}$. We pass this function to the built-in optimizer \texttt{fminunc}, which performs the BFGS update in the variable $\pmb{\eta}$. Finally, we recover $\mathbf{f} = H_x^{-1/2} \pmb{\eta}$.


\begin{algorithm}
\caption{Function \texttt{compute\_objective\_and\_gradient}$(\pmb{\eta})$}
\label{alg:objgrad}
\begin{algorithmic}
\STATE{Compute the initial displacement vector $\mathbf{f} := H_x^{-1/2}\pmb{\eta}$.}
\STATE{Solve forward problem for $\mathbf{U}$ using the initial displacement vector $\mathbf{f}$.}
\STATE{Solve adjoint problem for $\pmb{\Lambda}$ using the forward solution $\mathbf{U}$.}
\STATE{Compute the objective $\mathbf{J}(\mathbf{U})$ using \eqref{fullydiscretegaussianobj}.}
\STATE{Compute the gradient $\nabla_{\mathbf{f}}\mathbf{J}(\mathbf{U})$ using \eqref{fullydiscretegradient}.}

\RETURN{$\mathbf{J}(\mathbf{U})$ and $H_x^{1/2}\nabla_{\mathbf{f}}\mathbf{J}$.}
\end{algorithmic}
\end{algorithm}



 







\section{Forward numerical experiments}
\label{appendix:forward_numerical_experiments}

In this section, we present forward numerical simulations in one and two spatial dimensions using a Gaussian initial displacement field. We first consider the IBVP \eqref{contwaveeq}, \eqref{initcond}, \eqref{wavebcs} in the spatial domain $\Omega=[-1,1]$, and then its two-dimensional analogue in the spatial domain $\Omega=[-1,1]^2$. In both cases, the time interval is $0\le t\le T=2$ with zero source term $S = 0$. We set the constant wave speed $c=1$ and consider the undamped wave equation $\sigma=0$ and the damped wave equation $\sigma=1$.

\subsection{One-dimensional {Gaussian} forward problem} \label{gauss_1d_forward_section}

We choose the initial conditions
    $u(x,0) = e^{-100x^2} \quad \text{and} \quad \left.\frac{\partial u}{\partial t}\right|_{t=0} = 0,$
which correspond to a Gaussian profile for the initial displacement field and  zero initial velocity, respectively. We consider homogeneous Neumann boundary conditions, which are obtained by setting the boundary parameters
$\beta_L = \beta_R = 0$ and $\alpha_L = \alpha_R = 1$ with zero boundary data $b_L = b_R = 0$.

We compute the numerical solution using the 4th order space-time DP-SBP operators with the grid parameters $\Delta x = \Delta t = 0.02$. We use the second time derivative operator $D_{tt} = D^-_tD^-_t$. Snapshots of the numerical solution are shown in Fig. \ref{fig:gaussian1D_neumann_undamped} for the undamped wave equation $\sigma = 0$ and Fig. \ref{fig:gaussian1D_neumann_damped} for the damped wave equation $\sigma = 1$.

\begin{figure}[htpb]
  \centering
    \begin{subfigure}{.23\textwidth}
    \centering
    \includegraphics[width=\textwidth]{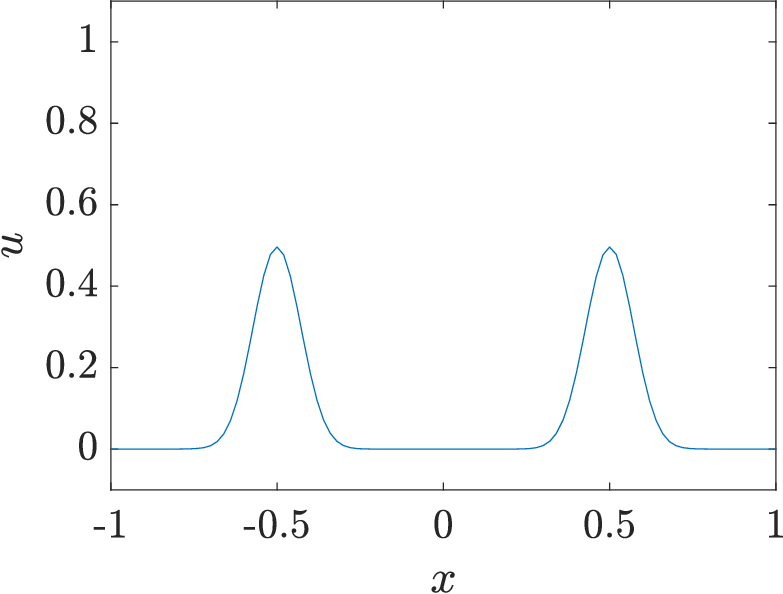}
    \caption{}
  \end{subfigure}
  \hfill
  \begin{subfigure}{.23\textwidth}
    \centering
    \includegraphics[width=\textwidth]{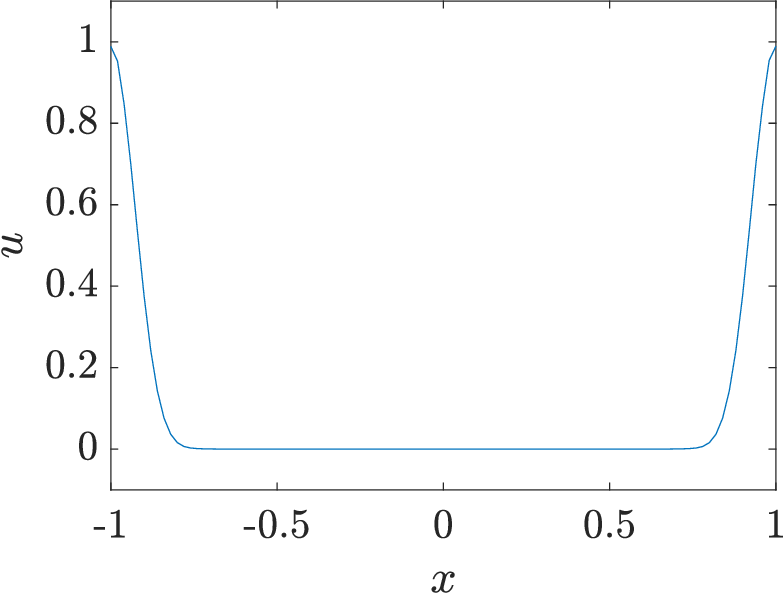}
    \caption{}
  \end{subfigure}
  \hfill
  \begin{subfigure}{.23\textwidth}
    \centering
    \includegraphics[width=\textwidth]{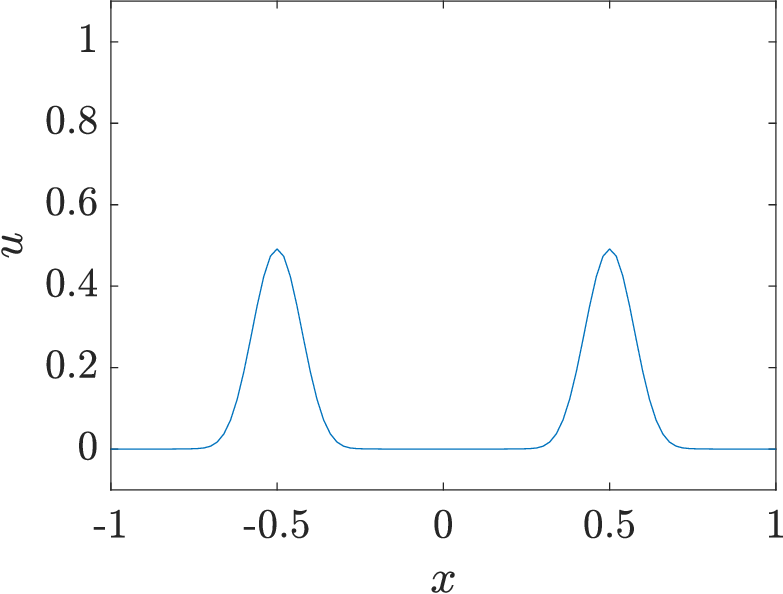}
    \caption{}
  \end{subfigure}
  \hfill
  \begin{subfigure}{.23\textwidth}
    \centering
    \includegraphics[width=\textwidth]{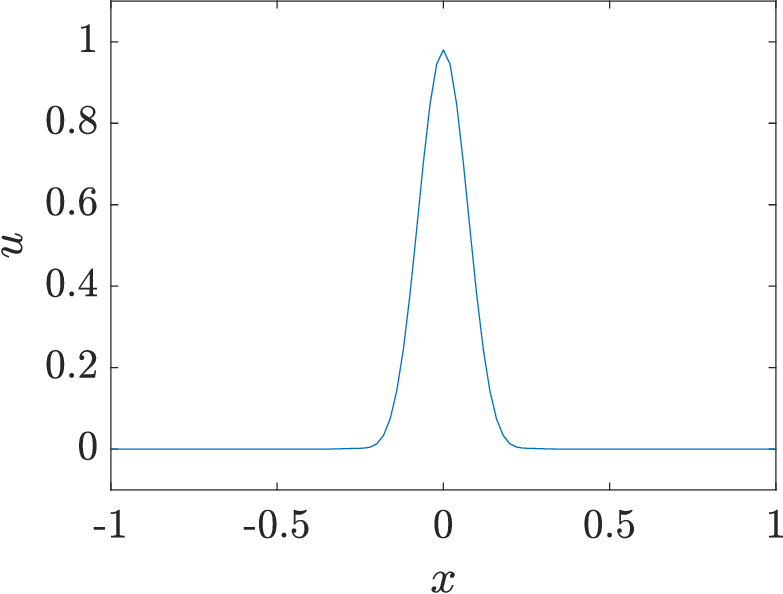}
    \caption{}
  \end{subfigure}

\caption{
  Numerical snapshots for the 1D Gaussian initial condition with $\sigma = 0$ and homogeneous Neumann boundary conditions at different times: (a) $t = 0.5$, (b) $t = 1$, (c) $t = 1.5$, and (d) $t = 2$. Note that there is no visible decay of the amplitude of the Gaussian profile at the final time $t = 2$.
}
  \label{fig:gaussian1D_neumann_undamped}
\end{figure}

\begin{figure}[htpb]
  \centering
    \begin{subfigure}{.23\textwidth}
    \centering
    \includegraphics[width=\textwidth]{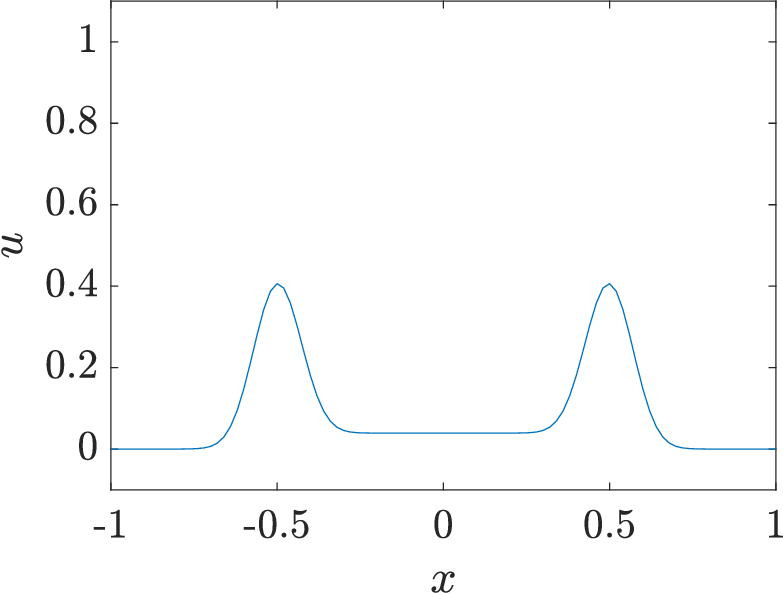}
    \caption{}
  \end{subfigure}
  \hfill
  \begin{subfigure}{.23\textwidth}
    \centering
    \includegraphics[width=\textwidth]{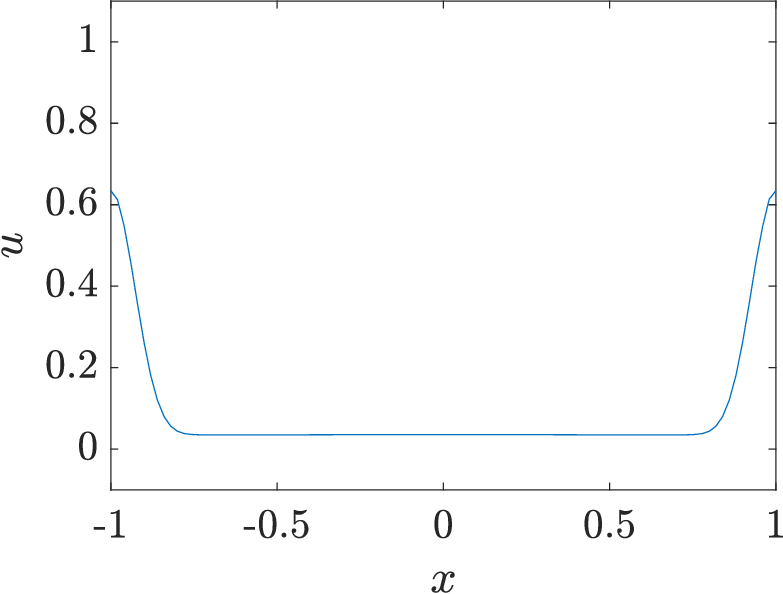}
    \caption{}
  \end{subfigure}
  \hfill
  \begin{subfigure}{.23\textwidth}
    \centering
    \includegraphics[width=\textwidth]{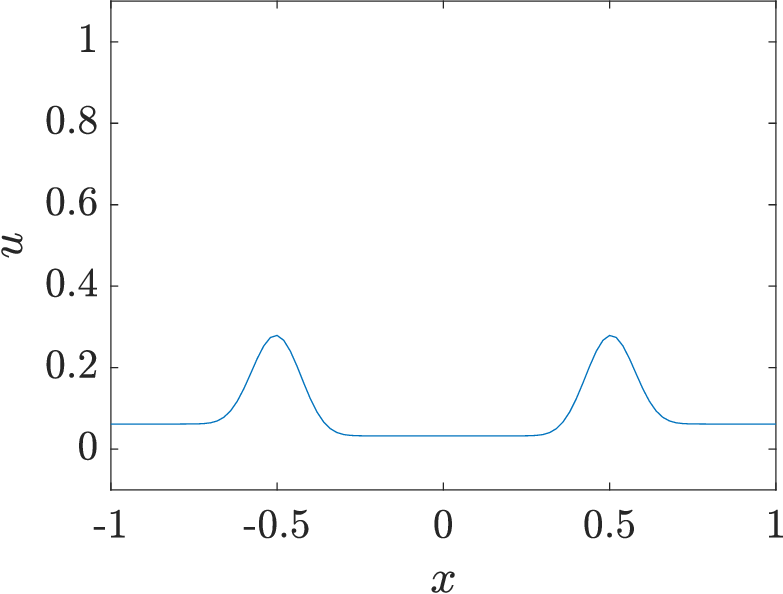}
    \caption{}
  \end{subfigure}
  \hfill
  \begin{subfigure}{.23\textwidth}
    \centering
    \includegraphics[width=\textwidth]{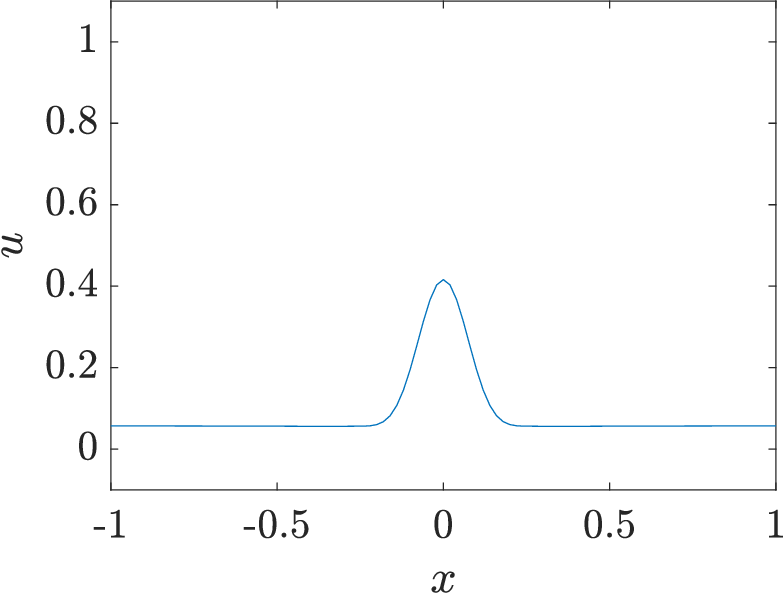}
    \caption{}
  \end{subfigure}
\caption{
  Numerical snapshots for the 1D Gaussian initial condition with $\sigma = 1$ and homogeneous Neumann boundary conditions at different times: (a) $t = 0.5$, (b) $t = 1$, (c) $t = 1.5$, and (d) $t = 2$. Note that there is significant decay of the amplitude of the Gaussian profile at the final time $t = 2$.
}
  \label{fig:gaussian1D_neumann_damped}
\end{figure}

\subsection{Two-dimensional Gaussian forward problem} \label{gauss_2d_forward_section}

 We choose the initial conditions 
    $u(x,y,0) = e^{-8(x^2+y^2)}$  \text{and} $\left. \frac{\partial u}{\partial t}\right|_{t=0} = 0$,
which correspond to a two-dimensional Gaussian profile for the initial displacement field and zero initial velocity, respectively. We consider homogeneous Neumann boundary conditions on $\partial \Omega$, analogous to those in subsection \ref{gauss_1d_forward_section}, which are obtained by setting the boundary parameters
$\beta_L = \beta_R = 0$ and $\alpha_L = \alpha_R = 1$ with zero boundary data.

The numerical solution is computed using the 4th order space-time DP-SBP operators with the space-time grid parameters $\Delta x = \Delta y = \Delta t = 0.0625$. The time interval $[0,2]$ is divided into two uniform blocks, and we apply the multi-block in time approach with the second time derivative operator $D_{tt} = D^-_tD^-_t$. Snapshots of the numerical solution are shown in Fig. \ref{fig:gaussian2D_neumann_undamped} for the 2D undamped wave equation $\sigma = 0$ and Fig. \ref{fig:gaussian2D_neumann_damped} for the 2D damped wave equation  $\sigma = 1$.

\begin{figure}[htpb]
  \centering
    \begin{subfigure}{.23\textwidth}
    \centering
    \includegraphics[width=\textwidth]{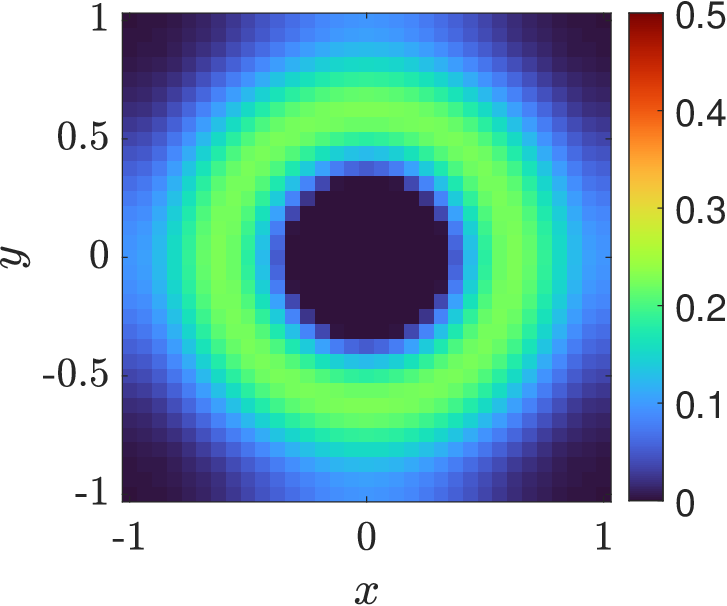}
    \caption{}
  \end{subfigure}
  \hfill
  \begin{subfigure}{.23\textwidth}
    \centering
    \includegraphics[width=\textwidth]{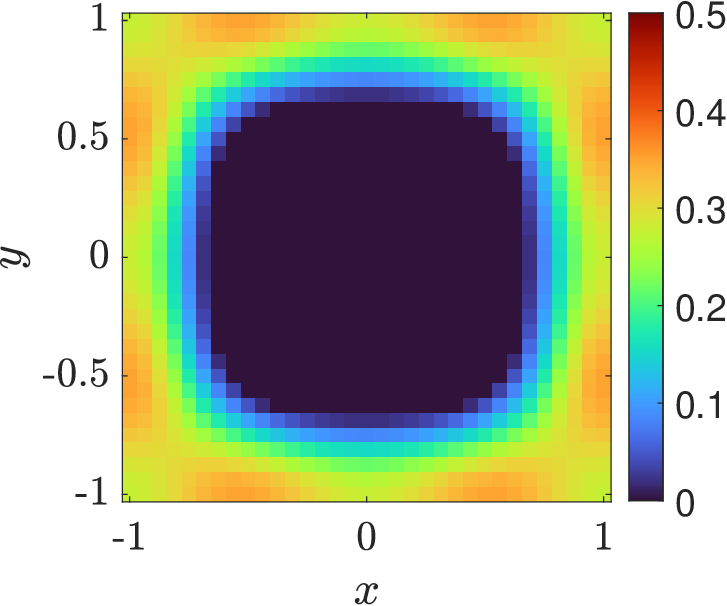}
    \caption{}
  \end{subfigure}
  \hfill
  \begin{subfigure}{.23\textwidth}
    \centering
    \includegraphics[width=\textwidth]{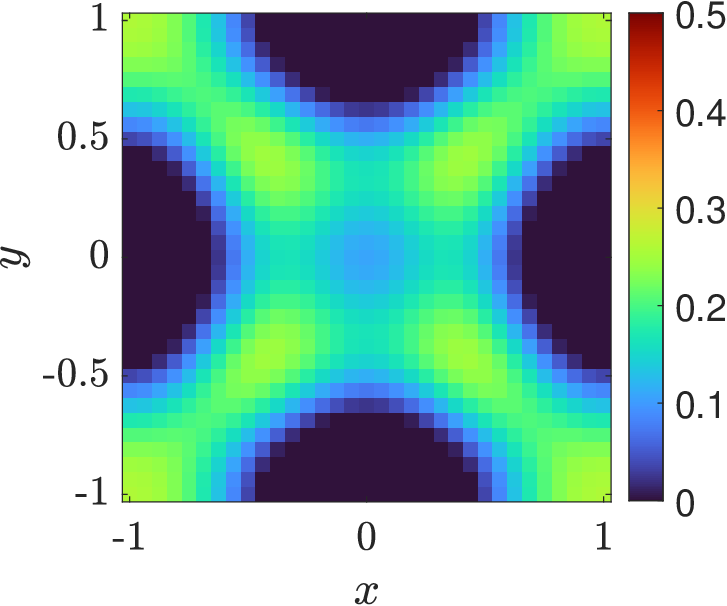}
    \caption{}
  \end{subfigure}
  \hfill
  \begin{subfigure}{.23\textwidth}
    \centering
    \includegraphics[width=\textwidth]{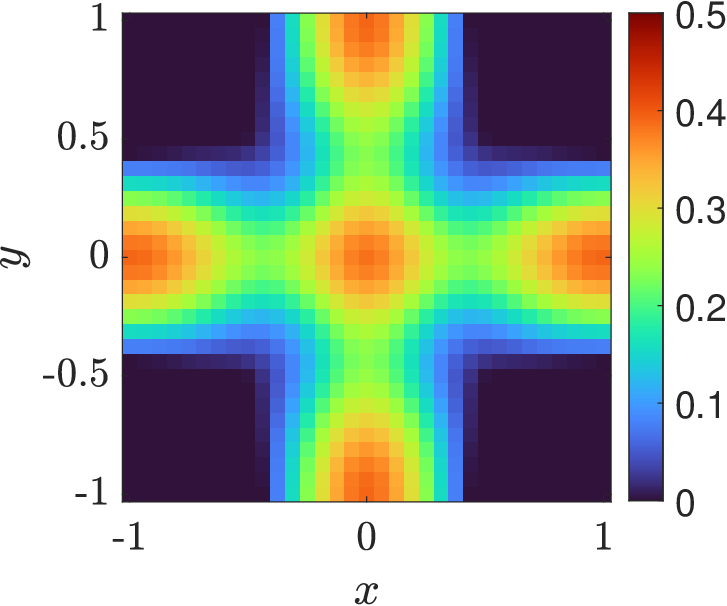}
    \caption{}
  \end{subfigure}

\caption{
  Numerical snapshots for the 2D Gaussian initial condition with $\sigma = 0$ and homogeneous Neumann boundary conditions at different times: (a) $t = 0.5$, (b) $t = 1$, (c) $t = 1.5$, and (d) $t = 2$.
}
  \label{fig:gaussian2D_neumann_undamped}
\end{figure}
\begin{figure}[htpb]
  \centering
    \begin{subfigure}{.23\textwidth}
    \centering
    \includegraphics[width=\textwidth]{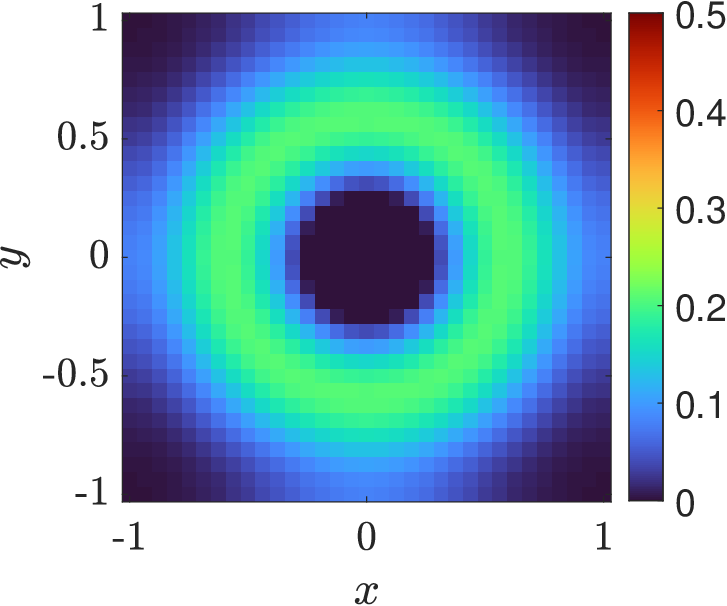}
    \caption{}
  \end{subfigure}
  \hfill
  \begin{subfigure}{.23\textwidth}
    \centering
    \includegraphics[width=\textwidth]{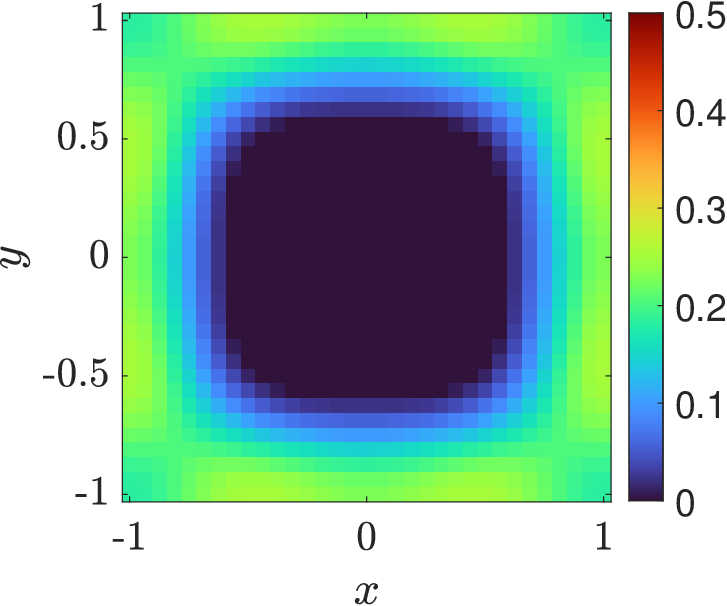}
    \caption{}
  \end{subfigure}
  \hfill
  \begin{subfigure}{.23\textwidth}
    \centering
    \includegraphics[width=\textwidth]{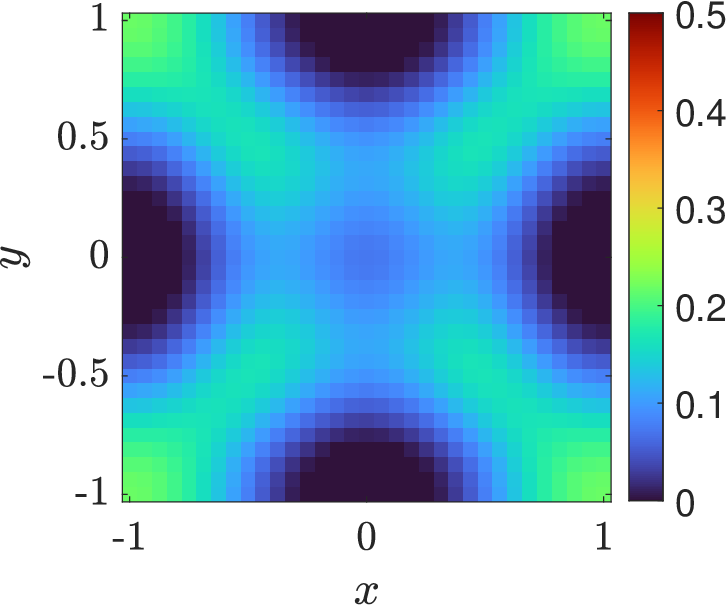}
    \caption{}
  \end{subfigure}
  \hfill
  \begin{subfigure}{.23\textwidth}
    \centering
    \includegraphics[width=\textwidth]{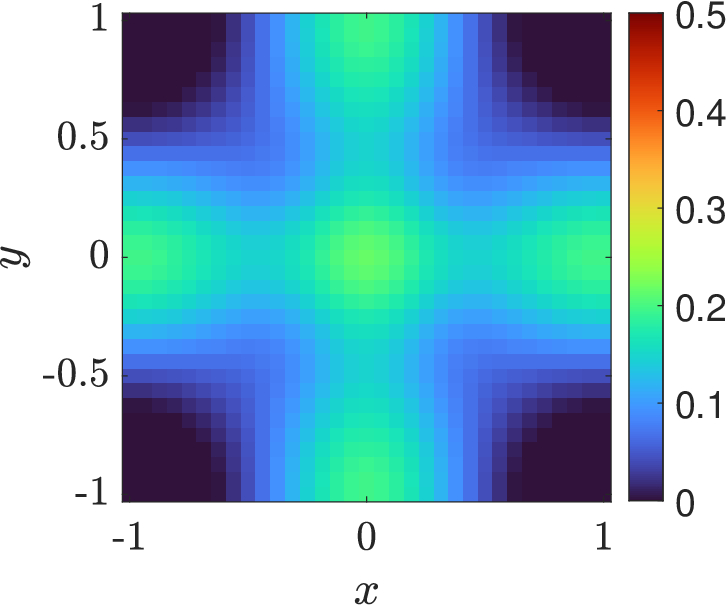}
    \caption{}
  \end{subfigure}

\caption{
  Numerical snapshots for the 2D Gaussian initial condition with $\sigma = 1$ and homogeneous Neumann boundary conditions at different times: (a) $t = 0.5$, (b) $t = 1$, (c) $t = 1.5$, and (d) $t = 2$.
}
  \label{fig:gaussian2D_neumann_damped}
\end{figure}

\section{Taylor remainder tests for the damped wave equations}
\label{app:damped-taylor-tests}

We present the Taylor remainders and observed convergence rates
for the one-dimensional and two-dimensional damped wave equations $\sigma = 1$. The results are shown in Tables
\ref{tab:taylor_test_1D_damped} and
\ref{tab:taylor_test_2D_damped}, respectively. As stated in subsections \ref{numerical_inverse_1D} and \ref{numerical_inverse_2D}, the observed convergence rates are similar to those in the undamped case $\sigma = 0$.

\begin{table}[htbp]
\centering
\begin{tabular}{c|cc|cc}
\hline
& \multicolumn{2}{c|}{DP-SBP}
& \multicolumn{2}{c}{RK4} \\
\cline{2-5}
$\varepsilon_i$
& $R(\varepsilon_i)$
& $\log_{10} \left(\tfrac{R(\varepsilon_{i-1})}{R(\varepsilon_i)} \right)$
& $R(\varepsilon_i)$
& $\log_{10} \left(\tfrac{R(\varepsilon_{i-1})}{R(\varepsilon_i)} \right)$ \\
\hline
$10^{-1}$
& $1.126845\times10^{-3}$
& ---
& $1.705375\times10^{-3}$
& --- \\

$10^{-2}$
& $1.126845\times10^{-5}$
& $2.0000$
& $4.764458\times10^{-5}$
& $1.5538$ \\

$10^{-3}$
& $1.126846\times10^{-7}$
& $2.0000$
& $3.535529\times10^{-6}$
& $1.1296$ \\

$10^{-4}$
& $1.126874\times10^{-9}$
& $2.0000$
& $3.412636\times10^{-7}$
& $1.0154$ \\

$10^{-5}$
& $1.130216\times10^{-11}$
& $1.9987$
& $3.400351\times10^{-8}$
& $1.0016$ \\

$10^{-6}$
& $1.544187\times10^{-13}$
& $1.8645$
& $3.399156\times10^{-9}$
& $1.0002$ \\
\hline
\end{tabular}
\caption{Taylor remainders and observed convergence rates for the inverse problem of recovering the 1D Gaussian initial displacement field for the damped wave equation $\sigma = 1$.}
\label{tab:taylor_test_1D_damped}
\end{table}

\begin{table}[htbp]
\centering
\begin{tabular}{c|cc|cc}
\hline
& \multicolumn{2}{c|}{DP-SBP}
& \multicolumn{2}{c}{RK4} \\
\cline{2-5}
$\varepsilon_i$
& $R(\varepsilon_i)$
& $\log_{10} \left(\tfrac{R(\varepsilon_{i-1})}{R(\varepsilon_i)} \right)$
& $R(\varepsilon_i)$
& $\log_{10} \left(\tfrac{R(\varepsilon_{i-1})}{R(\varepsilon_i)} \right)$ \\
\hline
$10^{-1}$
& $3.593547\times10^{-3}$
& ---
& $6.107453\times10^{-3}$
& --- \\

$10^{-2}$
& $3.593547\times10^{-5}$
& $2.0000$
& $6.123537\times10^{-5}$
& $1.9989$ \\

$10^{-3}$
& $3.593548\times10^{-7}$
& $2.0000$
& $6.284375\times10^{-7}$
& $1.9887$ \\

$10^{-4}$
& $3.593638\times10^{-9}$
& $2.0000$
& $7.892817\times10^{-9}$
& $1.9010$ \\

$10^{-5}$
& $3.603526\times10^{-11}$
& $1.9988$
& $2.398374\times10^{-10}$
& $1.5173$ \\

$10^{-6}$
& $4.627491\times10^{-13}$
& $1.8914$
& $1.855306\times10^{-11}$
& $1.1115$ \\
\hline
\end{tabular}
\caption{Taylor remainders and observed convergence rates for the inverse problem of recovering the 2D Gaussian initial displacement field for the damped wave equation $\sigma = 1$.}
\label{tab:taylor_test_2D_damped}
\end{table}

\bibliographystyle{unsrtnat}
\bibliography{references}  






\end{document}